\documentclass{article}
\usepackage{amsfonts, amssymb}
\usepackage{amsmath}
\usepackage{amsthm}
\usepackage{graphicx}
\usepackage{tikz}

\usepackage{color} 
   \definecolor{cites}{rgb}{0.50 , 0.00 , 0.00}  
   \definecolor{urls} {rgb}{0.00 , 0.00 , 0.50}  
   \definecolor{links}{rgb}{0.00 , 0.00 , 0.50}   
\definecolor{sccol}{rgb}{0,0,0.5}

\providecommand{\blue}{\textcolor{blue}}

\usepackage{enumitem}

\usepackage{etoolbox}
\makeatletter
\patchcmd{\@bibitem}{\ignorespaces}{\label{bib-#1}\ignorespaces}{}{}
\makeatother

\usepackage{hyperref} 
\hypersetup{  
      colorlinks=true,   
      citecolor=cites,   
      urlcolor=urls,     
      linkcolor=links,   
      pdftitle={Localisation of pseudospectra on groups},  
      pdfauthor={Simon N. Chandler-Wilde, Marko Lindner, Christian Seifert}, 
      pdfstartview=FitH,       
      bookmarksopen=false      
}

\usepackage{soul}

\newcommand\C{{\mathbb C}}

\newcommand\N{{\mathbb N}}
\newcommand\I{{\mathbb I}}
\newcommand\J{{\mathbb J}}

\newcommand\Q{{\mathbb Q}}
\newcommand\R{{\mathbb R}}
\newcommand\W{{\mathcal W}}
\newcommand\w{{w}}
\newcommand\Z{{\mathbb Z}}

\newcommand\ph{{\varphi}}
\newcommand\eps{{\varepsilon}}
\newcommand\specn{{\rm spec}}   
\newcommand\speps{{\rm spec}_\eps}

\newcommand\Specn{{\rm Spec}}   
\newcommand\Spec{{\rm Spec}\,}  
\newcommand\Speps{{\rm Spec}_\eps}

\newcommand\dist{{\rm dist}}

\newcommand\clos{{\rm clos\,}}

\newcommand\supp{{\rm supp\,}}
\newcommand\im{{\rm im\,}}
\newcommand\spn{{\rm span}}

\newtheorem{theorem}{Theorem}[section]
\newtheorem{lemma}[theorem]{Lemma}
\newtheorem{corollary}[theorem]{Corollary}
\newtheorem{proposition}[theorem]{Proposition}

\newenvironment{remark}
 {\par\noindent\refstepcounter{theorem}{\bf Remark \thetheorem\ }}
 {\raisebox{1mm}{\framebox{}}\par\pagebreak[2]}

\newenvironment{example}
 {\par\noindent\refstepcounter{theorem}{\bf Example \thetheorem\ }}
 {\raisebox{1mm}{\framebox{}}\par\pagebreak[2]}

\newcommand\Proofend{\rule{2mm}{2mm}}

\renewenvironment{proof}
 {\par\noindent{\bf Proof.}}
 {\Proofend\par\pagebreak[2]}

\numberwithin{equation}{section}   
\numberwithin{figure}{section}  

% ...

\begin{document}

\title{\bf Localisation of pseudospectra on discrete groups}
\author{{\sc Simon N. Chandler-Wilde},\quad {\sc Marko Lindner}\quad and\quad {\sc Christian Seifert}}
\date{\today}
\maketitle
\begin{quote}
\renewcommand{\baselinestretch}{1.0}
\footnotesize {\sc Abstract.} 
In this paper we generalise two of the methods and corresponding results from our previous paper ``On spectral inclusion sets and computing the spectra and pseudospectra of bounded linear operators'' [J. Spectr. Theory 14 (2024), 719–804] from tridiagonal operators on $\ell^2(\Z)$ to band operators $A$ on $\ell^2(G,Y)$ with a countable Abelian group $G$ and a Hilbert space $Y$. Again, we cover the pseudospectra of $A$, with error-control, via a union of pseudospectra of finite and moderately sized ``local patches'' of $A$. %
While a major application is to understand the case $G=\Z^d$ that is immanent in many physical problems, our new approach to the so-called $\tau$ and $\tau_1$ methods immediately extends to countable Abelian groups $G$. 
\end{quote}


\section{Introduction and overview}
The computation of spectra and pseudospectra of bounded linear operators is an important but challenging task with applications across science and engineering. Our aim is to compute spectral and pseudospectral enclosures via pseudospectra of finite so-called local patches of the operator, where the full operator $A$ acts on a Hilbert space-valued version of $\ell^2(G)$ with a discrete Abelian group $G$. This reduction to finite patches of $A$ is particularly useful in (but not limited to) the case when $G$ is an infinite group and $A$ corresponds to an infinite matrix $(a_{ij})_{i,j\in G}$. We quantify the sharpness of our enclosure sets and prove convergence to the spectrum, resp.~pseudospectrum, of $A$ as the size of the finite patches goes to infinity.

\medskip\noindent
{\bf Local methods and local patches.\ }
In \cite{CW.Heng.ML:JST} we present three ``local methods'' to compute sets that are guaranteed to include the pseudospectrum -- with known error bounds. By that phrase we mean methods that derive information on the spectrum and pseudospectra of an operator $A$ from many finite and moderately sized ``local patches'' of $A$. 

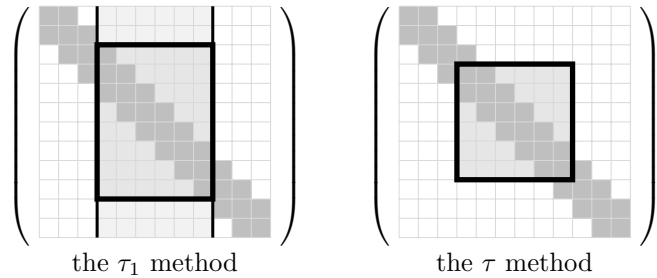
\begin{figure}[h]
\centering
\begin{tabular}{cp{0mm}cp{0mm}c}
$
\begin{pmatrix}~
\begin{tikzpicture}[scale=0.51]
\fill[gray!10] (1.5,6) rectangle ++(3,-6);
\fill[gray!20] (1.5,5) rectangle ++(3,-4);
\foreach \k in {0,0.5,...,5}
  \fill[gray!50] (\k,6-\k) rectangle ++(1.0,-1.0);
\draw[gray!30,very thin,step=0.5cm] (0,0) grid (6,6);
\draw[black,line width=0.7mm] (1.5,5) rectangle ++(3,-4);
\draw[black,line width=0.4mm] (1.5,6) -- ++(0,-6);
\draw[black,line width=0.4mm] (4.5,6) -- ++(0,-6);
\end{tikzpicture}
~\end{pmatrix}
$
&&
$
\begin{pmatrix}~
\begin{tikzpicture}[scale=0.51]
\fill[gray!20] (1.5,4.5) rectangle ++(3,-3);
\foreach \k in {0,0.5,...,5}
  \fill[gray!50] (\k,6-\k) rectangle ++(1.0,-1.0);
\draw[gray!30,very thin,step=0.5cm] (0,0) grid (6,6);
\draw[black,line width=0.7mm] (1.5,4.5) rectangle ++(3,-3);
\end{tikzpicture}
~\end{pmatrix}
$
\\
the $\tau_1$ method
&&
the $\tau$ method
\end{tabular}
\caption{The finite matrices that arise in the so-called $\tau_1$ and $\tau$
methods in \cite{CW.Heng.ML:JST}.} \label{fig:matrices}
\end{figure}

The difference between our three local methods, termed the $\tau_1$, $\tau$ and $\pi$ method in \cite{CW.Heng.ML:JST}, is what is meant by ``local patches'' of the operator. In this paper, we will only consider the $\tau_1$ and the $\tau$ method.
In 1D, the $\tau_1$ method (the $\tau_1$ standing for one-sided truncation) restricts the input vector to an interval of length $n$ before applying the operator, and the $\tau$ method ($\tau$ standing for the usual two-sided truncation) additionally restricts the output to the same interval of length $n$.
In matrix language, $\tau$ ends up with $n\times n$ matrices as their operator patches, and $\tau_1$ produces $\infty\times n$ matrices.
%
In Figure~\ref{fig:matrices} (that we essentially copy here from \cite{CW.Heng.ML:JST}) we summarise these two approaches figuratively for the tridiagonal case.

\medskip\noindent
{\bf Matrices and block entries.\ }
An operator $A$ on $\ell^2(G,Y)$ with a countable group $G$ and a Hilbert space $Y$ acts on vectors $x=(x_i)_{i\in G}\in\ell^2(G,Y)$ via multiplication by a matrix $(a_{ij})_{i,j\in G}$ that we will often identify with $A$. Note that the matrix entries are operators, i.e.\ $a_{ij}\in L(Y)$ for all $i,j\in G$.
The local patches of $A$ that we mention above are submatrices 
of that matrix $(a_{ij})_{i,j\in G}$.



In theory, there are infinitely many local patches of a fixed size $n$ in an infinite matrix; in fact, this infinite set of patches has a finite $\eps$-net for every $\eps>0$ if $\dim Y<\infty$; and in practice, this potentially infinite set is often finite (and sometimes known).

Note that, while the main analysis in \cite{CW.Heng.ML:JST} is for tridiagonal operators on $\ell^2(\Z)$, most results generalise to the vector valued $\ell^2$ space $\ell^2(\Z,Y)$. It is exactly this extra degree of freedom, the choice of the space $Y$, that enables \cite{CW.Heng.ML:JST} to deal with general band matrices via identification with a block tridiagonal matrix. Here we study the band case and even some band-dominated cases directly, already in case $Y=\C$.

\medskip\noindent
{\bf Localising global quantities.\ }
In \cite{CW.Heng.ML:JST}, for tridiagonal operators $A$ on $\ell^2(\Z)$, we show that, given a patch size $n\in\N$, for every nonzero $x\in\ell^2(\Z)$, there exists a (typically unknown) position $k\in\Z$, where the $n$-sized ``local patches'' $ x_{n,k}$ of $x$ and $A_{n,k}$ of $A$ satisfy\footnote{Note that the matrix entries of $A$ are denoted by lower case $a_{i,j}$ to avoid confusion with the local patches $A_{n,k}$.}
\begin{equation} \label{eq:patchAx_intro}
\frac{\|A_{n,k}\, x_{n,k}\|}{\| x_{n,k}\|}\ \le\ \frac{\|Ax\|}{\|x\|}\ +\ \eps_n
\end{equation}
and where $\eps_n\sim 1/n$ (with the proportionality constant quantified in \cite{CW.Heng.ML:JST} and shown to be optimal for several examples) is a result of truncation that we term the {\sl truncation penalty}.

The operator patches $A_{n,k}$ and the truncation penalty $\eps_n$ both depend on the method chosen, which is why we will write $A_{n,k}^{(M)}$ and $\eps_n^{(M)}$ with $M\in\{\tau_1,\tau\}$ when confusion needs to be avoided.

Evaluating \eqref{eq:patchAx_intro} for an $x$ that (almost) minimises $\frac{\|Ax\|}{\|x\|}$, one can draw conclusions about the norm of the inverse, resolvent norms, and pseudospectra of $A$ versus those of its local patches $A_{n,k}$. The fact that $k\in\Z$ in \eqref{eq:patchAx_intro} is unknown is reflected by taking the supremum of resolvent norms, resp.~the union of pseudospectra, of $A_{n,k}$ over all $k\in\Z$, leading to the promised inclusion sets of the form
\begin{equation} \label{eq:incl_intro}
\speps A\ \subset\ \bigcup_{k\in\Z} \specn_{\eps+\eps_n} A_{n,k}
\end{equation}
for the $\tau$ method. In the $\tau_1$ method also patches of the adjoint $A^*$ will join that formula \eqref{eq:incl_intro}, see below. Inclusion formulas like \eqref{eq:incl_intro} are useful in finding regions of the complex plane where there is, provably, no pseudospectrum and hence, no spectrum. More about pseudospectra in Section \ref{sec:tools}.

\medskip\noindent
{\bf Inclusions from both sides and convergence: the $\tau_1$ method.\ }
In the case of the $\tau_1$ method, abbreviating
\[
\bigcup_{k\in\Z} \Big(\speps A_{n,k}\ \cup\ \speps (A^*)_{n,k}\Big)\ =:\ \Gamma_{n,\eps}(A),
\]
we can complement the $\tau_1$ analogue of \eqref{eq:incl_intro}, leading to a computable set sandwich 
\begin{equation} \label{eq:sandw}
\Gamma_{n,\eps}(A)\ \subset\ \speps A\ \subset\ \Gamma_{n,\eps+\eps_n}(A)
\end{equation}
of $\speps A$  that is even convergent to $\speps A$ in Hausdorff distance as $n\to\infty$.
Indeed, iterating the statement of \eqref{eq:sandw}, getting
\[
\specn_{\eps-\eps_n} A\ \subset\ \Gamma_{n,\eps}(A)\ \subset\ \speps A\ \subset\ \Gamma_{n,\eps+\eps_n}(A)\ \subset\ \specn_{\eps+\eps_n} A,\qquad \eps>\eps_n,
\]
we argue by $\eps_n\to 0$ as $n\to\infty$ and continuity of the map $\eps\mapsto\speps A$ in the Hausdorff distance.

\medskip\noindent
{\bf Our focus.\ }
In \cite{CW.Heng.ML:JST} we moreover study
\begin{itemize} \itemsep-1mm
\item further approximations of the pseudospectrum and in particular of the spectrum of $A$,
\item the efficient reduction from infinitely many to finitely many local patches,
\item extensions to {\sl band-dominated operators} (norm-limits of sequences of band operators),
\item the computational cost in terms of the SCI of Hansen, Colbrook et al \cite{SCIshort,SCI},
\item arguments in favour of one or the other method,
\item and lots of examples.
\end{itemize}

We keep the current paper comparably short, touching only some of these aspects but referring to \cite{CW.Heng.ML:JST} for the others and many more. Our focus here is on the following improvements:
\begin{itemize} \itemsep-1mm
\item instead of $\ell^2(\Z)$, our operators now act on $\ell^2(G)$ with a countable Abelian group $G$ in case $M\in\{\tau_1,\tau\}$,
\item our band and even band-dominated operators $A$ are being studied directly (potentially saving on the penalty $\eps_n$) and not via identification with block tridiagonal operators and their approximates.
\end{itemize}
Our new results are generalisations of \cite{CW.Heng.ML:JST} without losing sharpness: if applied to the immediate setting of \cite{CW.Heng.ML:JST}, that is $G=\Z$ and tridiagonal operators $A$, we recover \eqref{eq:patchAx_intro} and the corresponding conclusions of \cite{CW.Heng.ML:JST}, in both methods, $\tau_1$ and $\tau$, and with the same values for $\eps_n$.

\medskip\noindent
{\bf Related work.\ }
This work has been developping via \cite{HengThesis,BigQuest,PAMM} and of course \cite{CW.Heng.ML:JST}, all for the case when $G=\Z^d$ (mostly with $d=1$). Only in \cite{BigQuest} we allow $d>1$ through some very rough estimates on the commutator $[A,P_k]$ of $A$ and the operator $P_k$ of restriction to a cube located at the point~$k$. Also note our other follow-up article, \cite{CW.ML:finite}, to \cite{CW.Heng.ML:JST}, where we extend the results from $G=\Z$ to finite intervals in $\Z$ and hence to finite (but potentially large) square matrices $A$.

 Extensive use has been made of \eqref{eq:patchAx_intro}, e.g., in \cite{TriRand,BigQuest}. For Schrödinger operators, comparable arguments even go back to \cite{Elliot,AvronMoucheSimon}. Independently, very similar work was developed in \cite{HegeThesis,Hege24} with stunning applications in \cite{Hege22,HegeThesis}. The very recent Master thesis \cite{WittigMSc} of Mattes Wittig and his corresponding paper \cite{Wittig26} are combining the approaches of \cite{HegeThesis,Hege24} and of this current paper. Moreover, in \cite{ColbrookEmbreeFillman2024} the authors study optimal algorithms for various quantities of the spectrum of a bounded self-adjoint operator on Hilbert spaces, such as its Lebesgue measure or its fractal dimension, with applications to Schr\"odinger operators for quasicrystal models, by means of spectral covers.

Related in a wider sense are of course the many spectral inclusion results by Gershgorin arguments and numerical ranges (of $A$, its polynomials or of its resolvent). We refer to the introduction of \cite{CW.Heng.ML:JST} for an extensive review also of that work.

\medskip\noindent
{\bf Structure of the paper.\ }
After a brief recall of tools and notations in Section~\ref{sec:tools} and some examples of groups and their band and band-dominated operators in Section \ref{sec:ex}, we review Laplacians on discrete graphs in Section ~\ref{sec:graph_laplacians}, which is one of the extra tools we need. After that, we state and prove our main results in Section~\ref{sec:main} before we close with the more technical extensions in Section~\ref{sec:extensions}.

\section{Notations and tools} \label{sec:tools}
Let $\N,\Z,\Q,\R$ and $\C$ denote the sets of all natural, integer, rational, real and complex numbers, respectively. Given a Banach space $E$, let $L(E)$ denote the Banach algebra of bounded linear operators on $E$.

\medskip\noindent
{\bf Band and band-dominated operators on discrete groups and the Wiener algebra.\ }
We take a countable Abelian group $(G,+)$, $Y$ a Hilbert space and define two basic operators on $E:= \ell^2(G,Y)$ (i.e.\ we equip $G$ with the counting measure):
\begin{itemize}
\item the operator $M_b$ of multiplication by a function $b\in\ell^\infty(G,L(Y))$, mapping $x\in\ell^2(G,Y)$ to $b\cdot x\in\ell^2(G,Y)$, defined by $(b\cdot x)_i=b_i x_i$ for all $i\in G$, and 
\item the operator $V_j$ of shift by a $j\in G$, mapping $x$ to $V_jx$, with $(V_jx)_i=x_{i-j}$ for all $i\in G$. (We will use the same symbol $V_j$ for the shift by $j\in G$ on $\ell^2(G)$.)
\end{itemize}
%
Every sum
\begin{equation} \label{eq:A}
A\ :=\ \sum_{j\in J}M_{b^{(j)}}V_j,
\end{equation}
with a finite set $J\subset G$ and $b^{(j)}\in\ell^\infty(G,L(Y))$ for $j\in J$, is called a {\sl band operator} on $E=\ell^2(G,Y)$. The functions $b^{(j)}$ are called the {\sl coefficients} or the {\sl diagonals} of $A$.
%
Since  $b^{(j)}\in\ell^\infty(G,L(Y))$ for all $j\in J$ and $J$ is finite, $A\in L(E)$.

Let $BO(E)\subset L(E)$ denote the set of all band operators on $E$. Occasionally we will make use also of the operator class $BDO(E)$, the Banach algebra of all {\em band-dominated operators} on $E=\ell^2(G,Y)$, defined as the closure in $L(E)$ of $BO(E)$, see, e.g., \cite{LiBook}. For $A\in BO(E)$ let
\[
||A||_{\W}\ :=\ \sum_{j\in G}\|b^{(j)}\|_\infty,
\]
noting that $\|A\|\leq \|A\|_\W$,
and let $\W(E)$ denote the  so-called {\sl Wiener algebra}, the closure   of $BO(E)$ in the Wiener algebra norm $\|\cdot\|_\W$. Then $A\in \W(E)$ if and only if
\begin{equation} \label{eq:AW}
A\ =\ \sum_{j\in G}M_{b^{(j)}}V_j,
\end{equation}
for some sequence $(b^{(j)})_{j\in G} \subset \ell^\infty(G,L(Y))$ with $\|A\|_{\W} <\infty$, and $\W(E)$ is  
a proper subalgebra (a Banach algebra if equipped with $\|\cdot\|_\W$) of $BDO(E)$; see, e.g., \cite{LiBook}.

We will also make use, for $0\leq p\leq 1$, of the subspace $\W_p(E)\subset \W(E)$, defined as
\begin{equation} \label{eq:Wpdef}
\W_p(E) := \left\{A\in \W(E): \w_p(A) := \sum_{j\in G}\|b^{(j)}\|^p_\infty < \infty\right\},
\end{equation}
with the understanding that $a^0:= \lim_{p\to 0^+} a^p$, for $a\geq 0$, so that $\w_0(A) = \#\{j\in G: b^{(j)}\neq 0\}$, whence $\W_0(E)=BO(E)$. Clearly, $w_1(A)=\|A\|_{\W}$ and $\W_1(E)=\W(E)$.

\medskip\noindent
{\bf The adjoint operator.\ }
Even though $A$ acts between Hilbert spaces, we work with the Banach space adjoint $A^*$ (in terms of matrices: the transpose -- without conjugates) throughout. In particular, $(\lambda A)^*=\lambda A^*$ and $\Spec A^*=\Spec A$ -- no complex conjugation.

\medskip\noindent
{\bf The lower norm.\ }
For a discrete group $G$, a Hilbert space $Y$, two subsets $\I,\J\subset G$, a bounded linear operator $A:\ell^2(\I,Y)\to\ell^2(\J,Y)$ and a set $T\subset \I$, put
\[
\nu_T(A)\ :=\ \inf\{\|Ax\|:\|x\|=1,\ \supp x\subset T\},
\]
where, for $x=(x_i)_{i\in \I}\in\ell^2(\I,Y)$, $\supp x:=\{i\in \I:\; x_i\ne 0\}$. Clearly, it holds that
\begin{equation}\label{eq:numon}
\nu_{T}(A)\ \ge\ \nu_{U}(A)
\quad\text{if}\quad
T\subset U\subset \I.
\end{equation}
Another standard and basic result (e.g.~\cite[Lemma 2.38]{LiBook}) is that
\begin{equation} \label{eq:Lipschitz}
|\nu_T(A)-\nu_T(B)|\ \le\ \|A-B\|
\end{equation}
for all $T\subset \I$. Abbreviate $\nu_\I(A)=:\nu(A)$.
%
%
A key fact about $\nu(\cdot)$ in the case $\I=\J$ is that
\begin{equation} \label{eq:invnu}
\|A^{-1}\|^{-1}\ =\ \min\big(\,\nu(A),\nu(A^*)\,\big)\ =:\ \mu(A),
\end{equation}
where the expression is interpreted as $\infty^{-1}=0$ if and only if $A$ is not invertible.

\medskip\noindent
{\bf Spectrum and pseudospectra.\ }
Let $\I=\J$, i.e.~$A$ is a bounded linear operator $\ell^2(\I,Y)\to\ell^2(\I,Y)$.
By the singular case of \eqref{eq:invnu}, the {\sl spectrum of $A$} can be characterised via $\mu(\cdot)$:
\[
\Spec A\ :=\ \{\lambda\in\C\ :\ (A-\lambda I)\text{ is not boundedly invertible}\}\ \stackrel{\eqref{eq:invnu}}=\ \{\lambda\in\C\ :\ \mu(A-\lambda I)=0\}.
\]
For $\eps>0$, the set
\begin{equation} \label{eq:speps}
\speps A\ :=\ \{\lambda\in\C\ :\ \mu(A-\lambda I)<\eps\}\ \stackrel{\eqref{eq:invnu}}=\ \{\lambda\in\C\ :\ \|(A-\lambda I)^{-1}\|>1/\eps\}
\end{equation}
is the {\sl open $\eps$-pseudospectrum of $A$}, while, for $\eps\ge 0$,
\begin{eqnarray}
\Speps A &:=& \{\lambda\in\C\ :\ \mu(A-\lambda I)\le\eps\}\label{eq:Speps}\\
&\stackrel{\eqref{eq:invnu}}=&
\left\{\begin{array}{cl}
\{\lambda\in\C\ :\ \|(A-\lambda I)^{-1}\|\ge 1/\eps\}&\text{if } \eps>0,\\[1mm]
\Spec A&\text{if } \eps=0
\end{array}\right. \nonumber
\end{eqnarray}
is the {\sl closed $\eps$-pseudospectrum of $A$}.

The sets $\speps A$ and $\Speps A$ are indeed open, respectively closed. (They are the preimages of the open, resp. closed, sets $[0,\eps)$ and $[0,\eps]$ w.r.t.~the function $\lambda\mapsto \mu(A-\lambda I)$ that is Lipschitz continuous, by \eqref{eq:Lipschitz}.) By \cite{Globevnik74,Globevnik76,Shargo08,Shargo09} it holds that
\begin{equation} \label{eq:closspeps}
\Speps A\ =\ \clos\speps A,\qquad\eps>0.
\end{equation}

Our definitions of spectrum and pseudospectra for $A:\ell^2(\I,Y)\to\ell^2(\I,Y)$ include the cases of finite square matrices, e.g.~our $\tau$ patches $A_{n,k}$. Our $\tau_1$ method however leads to operator patches $\ell^2(T,Y)\to\ell^2(G,Y)$ for certain finite subsets $T\subset G$. 
(In particular, for a band operator $A$, the $\tau_1$ patches can be considered as acting $\ell^2(T,Y)\to\ell^2(T+J,Y)$, where $T+J$ is the Minkowski sum of $T$ and the finite set $J\subset G$ from \eqref{eq:A}.)

The concept of pseudospectra is well-studied also for operators of this kind, i.e.~for rectangular matrices with more rows than columns \cite{WrightTrefethen,TrefEmbBook}:
For a finite and proper subset $T\subset G$, denote
\[
I_T:\ell^2(T,Y)\to \ell^2(G,Y),\quad x\mapsto x
\]
and suppose $A:\ell^2(T,Y)\to\ell^2(G,Y)$.

In this setting, there is one case that needs extra attention: If also $\dim Y<\infty$ then the adjoint of $B:=A-\lambda I_T$, mapping $\ell^2(G,Y)\to\ell^2(T,Y)$, is never injective since $G$ is strictly larger than $T$, which is why $\nu(B^*)=0$ for all $\lambda\in\C$. In this particular case, one deviates from \eqref{eq:speps} and \eqref{eq:Speps}, instead putting
\begin{equation} \label{eq:speps.rect}
\left.\begin{array}{ccc}
\speps A\ :=\ \{\lambda\in\C\ :\ \nu(A-\lambda I_T)<\eps\},&&\eps>0,\\[1mm]
\Speps A\ :=\ \{\lambda\in\C\ :\ \nu(A-\lambda I_T)\le\eps\},&&\eps\ge 0.
\end{array}
\qquad\qquad\qquad\right\}
\end{equation}

In all other cases, square or $\dim Y=\infty$, $\speps A$ and $\Speps A$ are defined by \eqref{eq:speps} and \eqref{eq:Speps}.


\section{Examples: groups and their band(-dominated) operators} \label{sec:ex}
There is a large amount of countable Abelian groups, many of them infinitely generated like $(\Q,+)$, $(\Q\setminus\{0\},\cdot)$, most countable subgroups of $\Z^\N$ with componentwise addition (for example those with a finite $\ell^p$-norm for some fixed $p<\infty$) and direct sums thereof.

More canonic is the situation in the case of finitely generated (hence countable) Abelian groups $G$, where there is a finite set $F=\{f_1,\dots,f_m\}\subset G$ of so-called generators such that every $g\in G$ can be written as a finite sum of $f_1,\dots,f_m$ and their inverses, that is,
\[
g\ =\ \alpha_1f_1+\dots+\alpha_m f_m,\qquad \alpha_i\in\Z.
\]
For a generator $f\in F$, the set $\{\alpha f:\alpha\in\Z\}$ can either be a)~infinite (hence, a copy of $\Z$ in $G$) or b)~finite (hence, a cyclic subgroup of $G$). By the fundamental theorem of finitely generated Abelian groups, see e.g.~\cite[\S 9.7]{Cohn}
\begin{equation} \label{eq:fundthm}
G\ \cong\ \Z^d\ \oplus\ H,
\end{equation}
where $d\in\N\cup\{0\}$ and $H$ is a finite sum of cyclic groups $\Z/q\Z$ with $q\in\N$. If $F$ is minimal, the $\Z^d$ part is generated by $d$ generators $f\in F$ of type a) and $H$ is generated by $m-d$ generators of type b).

Let us look at the examples $G=\Z^d$ and $G=\Z/q\Z$ first before we come to the mixed case.

\medskip\noindent
{\bf Infinite groups of the form $G=\Z^d$.\ }
Many problems in solid state physics are modelled in $\Z^d$, other problems lead to $\Z^d$ after a suitable discretisation of $\R^d$. Band operators on $\ell^2(\Z^d)$ are, by definition, finite sums of finite products of shifts and multiplication operators, each multiplying by a bounded function $b\in\ell^\infty(\Z^d)$. The shift operator $V_g$ with $g\in\Z^d$ moves the data from location $i\in\Z^d$ to location $i+g\in\Z^d$. In case $d=1$ the matrix $(a_{ij})_{i,j\in\Z^d}$ of a band operator on $\ell^2(\Z^d)$ is a $\Z\times\Z$ matrix that is supported on finitely many diagonals -- a bi-infinite band matrix in the usual sense.

\medskip\noindent
{\bf Finite cyclic groups $G=\Z/q\Z$.\ } The cyclic group $\Z/q\Z$ with $q\in\N$ can be identified with a discrete circle with $q$ sites or with the set $\{1,2,\dots,q\}$ with addition modulo $q$. The vectors $u\in\ell^2(\Z/q\Z)$ can thus be imagined as sets of data located at the $q$ sites of that circle. Multiplication operators again act by pointwise multiplication, and the shift $V_g$ moves the data $g$ positions further along the circle.

The matrices $(a_{ij})_{i,j\in G}$ would have their entries indexed over the $q\times q$ torus $(\Z/q\Z)^2$. The latter is a bit unusual to envisage and to work with, which is why we favour
the identification of $G=\Z/q\Z$ with $\{1,\dots,q\}$ with addition modulo $q$. Operators on $\ell^2(\Z/q\Z)$ are then identified with usual $q\times q$ matrices featuring the cyclic structure of the group. For example, the matrix of the shift $V_1$ has ones on the first subdiagonal and an extra one in the north-east corner at position $(i,j)=(1,q)$. More generally, a band operator $A$ with bandwidth $b$ on $\ell^2(\Z/q\Z)$ is identified with a $q \times q$ matrix $(a_{ij})_{i,j=1}^k$ that is supported where $i-j$ is in the $b$-neighborhood of $0$ or $q$ or $-q$; so the usual band structure extends to the two off-diagonal corners of the matrix. 

It is perhaps worth mentioning that, hence, our methods can be used to bound and approximate the spectra and pseudospectra of large finite square matrices with a band structure that extends to the off-diagonal corners. 
For illustration, here is this cyclic situation with $q=6$ and $b=1$ (tridiagonal),
\begin{equation}
\left(\begin{array}{cccccc}
\beta_1&\gamma_2& &  & &{\alpha_6}\\
\alpha_1&\beta_2&\gamma_3\\
&\alpha_2&\beta_3&\gamma_4\\
&&\alpha_3&\beta_4&\gamma_5\\
&&&\alpha_4&\beta_5&\gamma_6\\
{\gamma_1}&&&&\alpha_{5}&\beta_6
\end{array}
\!\right),
\end{equation}
where $\alpha_i, \beta_i,\gamma_i\in\C$ for $i=1,\dots,6$ and where we highlight the two off-diagonal corner entries, \colorbox{black!7}{$\alpha_6$} and \colorbox{black!15}{$\gamma_1$}, to better illustrate the cyclic effects in the following list of all $\tau$ patches $A_{n,k}$ of size $n=3$ as $k=1,\dots,6$:
\[
\begin{array}{ccc}
A_{n,1}=\begin{pmatrix} \beta_2&\gamma_3\\ \alpha_2&\beta_3&\gamma_4\\ &\alpha_3&\beta_4\end{pmatrix},\quad~
& 
A_{n,2}=\begin{pmatrix} \beta_3&\gamma_4\\ \alpha_3&\beta_4&\gamma_5\\ &\alpha_4&\beta_5\end{pmatrix},\quad~
& 
A_{n,3}=
\left(
\begin{array}{ccc|}
\beta_4&\gamma_5&\\ \alpha_4&\beta_5&\gamma_6\\ &\alpha_5&\beta_6\\\hline
\end{array}
\right)
,
\\[6mm]
A_{n,4}=
\left(\!\!
\begin{array}{cc|c}
\beta_5&\gamma_6\\
\alpha_5&\beta_6&{\gamma_1}\\ \hline
&{\alpha_6}&\beta_1
\end{array}
\!\!\right)
,\quad~
& 
A_{n,5}=
\left(\!\!
\begin{array}{c|cc}
\beta_6&{\gamma_1}\\ \hline
{\alpha_6}&\beta_1&\gamma_2\\
&\alpha_1&\beta_2
\end{array}
\!\!\right)
,\quad~
& 
A_{n,6}=
\left(
\begin{array}{|ccc}
\hline
\beta_1&\gamma_2\\ \alpha_1&\beta_2&\gamma_3\\ &\alpha_2&\beta_3
\end{array}
\right)
.
\end{array}
\]
In this context we also want to mention our recent work \cite{CW.ML:finite} on the spectra and pseudospectra of finite square matrices with (or without) a band structure not extending to the off-diagonal corners.

\medskip\noindent
{\bf A mixed example.\ }
Having discussed the single ingredients of \eqref{eq:fundthm}, $\Z^d$ and $\Z/q\Z$, we finally look at a mixed example, precisely, at $G=\Z^2\oplus (\Z/2\Z)$ and its pretty geometric representation.

\begin{example} \label{ex:honeycomb}
The nodes of the so-called honeycomb lattice form an Abelian group $G$ with three generators: $f_1$ and $f_2$ are of type a) and $f_3$ generates a 2-cycle, see Figure \ref{fig:honeycomb}.

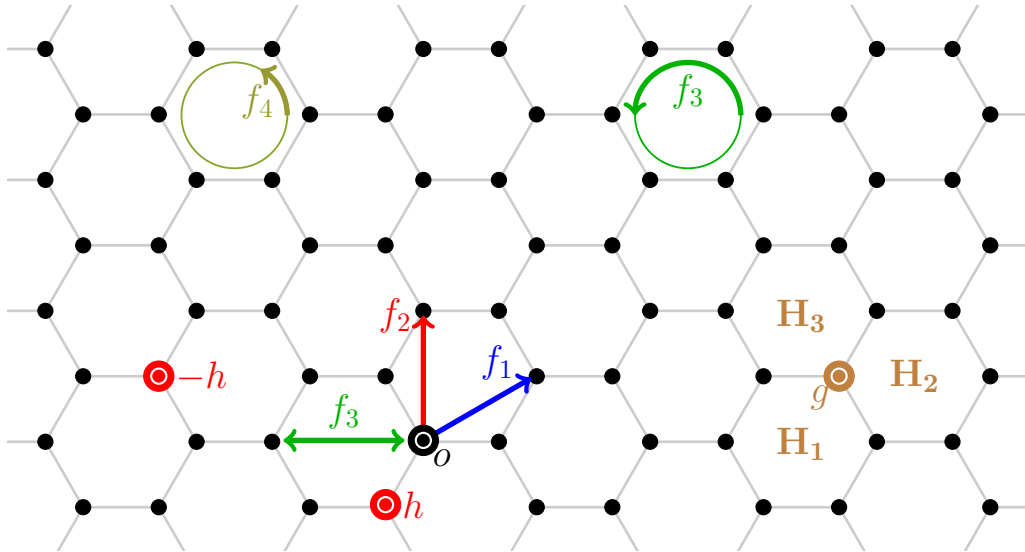
\begin{figure}[h]
\centering
\begin{tikzpicture}[scale=1.0]
\clip(0.5,0.3) rectangle (14,7.5);
\foreach \i in {0,...,4} {
  \foreach \j in {0,...,4} {
     \foreach \a in {0,120,240} \draw [black!20,line width=1pt] (3*\i,2*sin{60}*\j) -- +(\a:1);
     \foreach \a in {0,120,240} \draw [black!20,line width=1pt] (3*\i+3*cos{60},2*sin{60}*\j+sin{60}) -- +(\a:1);
  }
}
\foreach \i in {0,...,4} {
  \foreach \j in {0,...,4} {
     \draw[fill] (3*\i,2*sin{60}*\j) circle (0.1);
     \draw[fill] (3*\i+3*cos{60},2*sin{60}*\j+sin{60}) circle (0.1);
     \draw[fill] (1+3*\i,2*sin{60}*\j) circle (0.1);
     \draw[fill] (1+3*\i+3*cos{60},2*sin{60}*\j+sin{60}) circle (0.1);
  }
}
\draw[->,line width=2pt,blue] (6,1.75) -- +(30:1.65) node[xshift=-4.5mm,yshift=2mm] {\Large 
$f_1$};
\draw[->,line width=2pt,red] (6,1.75) -- +(90:1.65) node[left] {\Large 
$f_2$};
\draw[green!30!brown,line width=0.7pt] (3.5,6.05) circle (0.7);
\draw[->, line width=2pt, green!20!brown] (4.2,6.05) arc (0:60:0.7) node[near start, left] {\Large 
$f_4$};
\draw[<->, line width=2pt, green!70!black] (4.15,1.75) -- +(1.6,0) node[midway,above] {\Large 
$f_3$};
\draw[green!70!black,line width=0.7pt] (9.5,6.05) circle (0.7);
\draw[->, line width=2pt, green!70!black] (10.2,6.05) arc (0:180:0.7) node[midway,below] {\Large 
$f_3$};
\draw[fill,brown] (11.5,2.6) circle (0.2) node[below left] {\Large 
$g$};
\draw[white,thick] (11.5,2.6) circle (0.1);
\node[brown] at (11,1.7) {\Large $\bf H_1$};
\node[brown] at (12.5,2.6) {\Large $\bf H_2$};
\node[brown] at (11,3.4) {\Large $\bf H_3$};
\draw[fill,black] (6,1.75) circle (0.2) node[below right] {\Large 
$o$};
\draw[white,thick] (6,1.75) circle (0.1);
\draw[fill,red] (5.5,0.9) circle (0.2) node[right=1mm] {\Large 
$h$};
\draw[white,thick] (5.5,0.9) circle (0.1);
\draw[fill,red] (2.5,2.6) circle (0.2) node[right=1mm] {\Large 
$-h$};
\draw[white,thick] (2.5,2.6) circle (0.1);
\end{tikzpicture}
\caption{The honeycomb lattice $G$ from Example \ref{ex:honeycomb} with its three generators, $f_1, f_2$ and $f_3$. Note that we draw each $f_i$ here in terms of its action $z\mapsto f_i+z$ and not primarily as a point of $G$.\\Also shown are an arbitrary node $g\in G$ and the three hexagons $H_j$ that it is part of as well as an element $h\in G$ and its inverse $-h$ whose location seems unexpected from the perspective of the zero element $o\in G$.} \label{fig:honeycomb}
\end{figure}

In this case, \eqref{eq:fundthm} holds with $d=2$ and $H=\Z/2\Z$. One might be tempted to thinking that $H$ had to be a $6$-cycle, $\Z/6\Z$, generating (e.g., via $f_4$) the hexagon located at an integer combination of $f_1$ and $f_2$, concluding that $G$ were generated by $f_1,f_2$ and $f_4$ -- which is wrong:

Note that every node $g\in G$ is part of three different hexagons. For example, the bold node $g$ shown in Figure \ref{fig:honeycomb} is of the form $g=\alpha_1f_1+\alpha_2f_2+1f_4$, where $\alpha_1f_1+\alpha_2f_2$ is the eastern endpoint of the hexagon $H_1$. But $g$ is also $\beta_1f_1+\beta_2f_2+3f_4$ and also $\gamma_1f_1+\gamma_2f_2+5f_4$ with $(\beta_1,\beta_2)$ and $(\gamma_1,\gamma_2)$ addressing the eastern endpoints of hexagons $H_2$ and $H_3$, respectively, w.r.t.~$(f_1,f_2)$.

This ambiguity (by a factor of three) in the representation of $g\in G$ as an integer combination of $(f_1,f_2,f_4)$ resolves when we pass to the generating system $(f_1,f_2,f_3)$ instead, noting that $f_3=3f_4$ is rotation by $180^o$, i.e., flip between the eastern and western endpoint of a hexagon\footnote{Note that every $g\in G$ is the eastern or western endpoint of some -- and exactly one -- hexagon.}. Then $\beta_1f_1+\beta_2f_2+1f_3$ via hexagon $H_2$ is the only remaining valid representation of our bold node $g$ in Figure \ref{fig:honeycomb}. Obviously, $f_3$ generates the subgroup $\Z/2\Z=:H$, so that $G\cong \Z^2\oplus (\Z/2\Z)$. 

Having understood that $G$ is generated by $f_1,f_2$ and $f_3$, consider $h:=f_1-f_2+f_3$, so that $-h=-f_1+f_2-f_3=-f_1+f_2+f_3$ since $f_3=-f_3$. Both $h$ and $-h$, together with the zero element $o\in G$, are depicted in Figure \ref{fig:honeycomb}. Note how $o$ is not even on the line connecting $h$ to $-h$.

To better understand this apparent mismatch between algebra and geometry, let us look at the grid lines w.r.t.~$(f_1,f_2)$ as depicted in blue on the left in Figure \ref{fig:honeycomb2}. Half of the group $G$, precisely,
\[
\{\alpha_1 f_1+\alpha_2 f_2+\alpha_3 f_3:\alpha_1,\alpha_2\in\Z,\alpha_3\in 2\Z\}
\ =\ \{\alpha_1 f_1+\alpha_2 f_2:\alpha_1,\alpha_2\in\Z\}
\ =:\ \spn_\Z\{f_1,f_2\}
\]
sits on these blue grid lines. The other half of $G$, that is,
\[
\{\alpha_1 f_1+\alpha_2 f_2+\alpha_3 f_3:\alpha_1,\alpha_2\in\Z,\alpha_3\in 2\Z+1\}
\ =\ \{\alpha_1 f_1+\alpha_2 f_2+f_3:\alpha_1,\alpha_2\in\Z\},
\]
sits on the red grid lines on the right of Figure \ref{fig:honeycomb2}, which is a copy of the blue grid, $\spn_\Z\{f_1,f_2\}$, but shifted by $f_3$.
\begin{figure}[h]
\centering
\begin{tikzpicture}[scale=0.7]
\clip(2.3,0.3) rectangle (12.2,7.5);
\foreach \k in {1,...,7} {
   \draw[blue,line width=0.7pt] ({3*\k-12},0) -- +(17.3,10);
   \draw[blue,line width=0.7pt] ({1.5*\k+1.5},0) -- +(0,8);
}  
\foreach \i in {0,...,4} {
  \foreach \j in {0,...,4} {
     \foreach \a in {0,120,240} \draw [black!10,line width=1pt] (3*\i,2*sin{60}*\j) -- +(\a:1);
     \foreach \a in {0,120,240} \draw [black!10,line width=1pt] (3*\i+3*cos{60},2*sin{60}*\j+sin{60}) -- +(\a:1);
  }
}
\foreach \i in {0,...,4} {
  \foreach \j in {0,...,4} {
     \draw[fill] (3*\i,2*sin{60}*\j) circle (0.1);
     \draw[fill] (3*\i+3*cos{60},2*sin{60}*\j+sin{60}) circle (0.1);
     \draw[fill] (1+3*\i,2*sin{60}*\j) circle (0.1);
     \draw[fill] (1+3*\i+3*cos{60},2*sin{60}*\j+sin{60}) circle (0.1);
  }
}
\draw[->,line width=2pt,blue] (7.5,2.6) -- +(30:1.8);
\draw[->,line width=2pt,blue] (7.5,2.6) -- +(90:1.7);
\draw[fill,black] (7.5,2.6) circle (0.2) node[above=4mm, right] {\Large 
$o$};
\draw[white,thick] (7.5,2.6) circle (0.1);
\draw[fill,blue] (9,1.7) circle (0.2) node[above=4mm, right] {\Large 
$a$};
\draw[white,thick] (9,1.7) circle (0.1);
\draw[fill,blue] (6,3.5) circle (0.2) node[left=1mm] {\Large 
$-a$};
\draw[white,thick] (6,3.5) circle (0.1);
\draw[fill,green!70!black] (5.5,2.6) circle (0.2) node[below=5mm, left=-0.5mm] {\Large 
$f_3$};
\draw[white,thick] (5.5,2.6) circle (0.1);
\end{tikzpicture}
~ ~ ~ ~
\begin{tikzpicture}[scale=0.7]
\clip(2.3,0.3) rectangle (12.2,7.5);
\foreach \k in {1,...,7} {
   \draw[blue!30,line width=0.7pt] ({3*\k-12},0) -- +(17.3,10);
   \draw[blue!30,line width=0.7pt] ({1.5*\k+1.5},0) -- +(0,8);
   \draw[red,line width=0.7pt] ({3*\k-11},0) -- +(17.3,10);
   \draw[red,line width=0.7pt] ({1.5*\k+2.5},0) -- +(0,8);
}  
\foreach \i in {0,...,4} {
  \foreach \j in {0,...,4} {
  }
}
\foreach \i in {0,...,4} {
  \foreach \j in {0,...,4} {
     \draw[fill] (3*\i,2*sin{60}*\j) circle (0.1);
     \draw[fill] (3*\i+3*cos{60},2*sin{60}*\j+sin{60}) circle (0.1);
     \draw[fill] (1+3*\i,2*sin{60}*\j) circle (0.1);
     \draw[fill] (1+3*\i+3*cos{60},2*sin{60}*\j+sin{60}) circle (0.1);
  }
}
\draw[->,line width=2pt,blue] (7.5,2.6) -- +(30:1.8);
\draw[->,line width=2pt,blue] (7.5,2.6) -- +(90:1.7);
\draw[fill,black] (7.5,2.6) circle (0.2) node[above=4mm, right] {\Large 
$o$};
\draw[white,thick] (7.5,2.6) circle (0.1);
\draw[fill,green!70!black] (5.5,2.6) circle (0.2) node[below=5mm, left=-0.5mm] {\Large 
$f_3$};
\draw[white,thick] (5.5,2.6) circle (0.1);
\draw[fill,red] (7,1.7) circle (0.2) node[below=5mm, left=-0.5mm] {\Large 
$b$};
\draw[white,thick] (7,1.7) circle (0.1);
\draw[fill,red] (4,3.5) circle (0.2) node[below=5mm, left=-0.5mm] {\Large 
$-b$};
\draw[white,thick] (4,3.5) circle (0.1);
\end{tikzpicture}

\caption{On the left we see in blue the integer grid generated by $(f_1,f_2)$ and the geometric relation between an element $a$ and its inverse. On the right we see in light blue a copy of the blue grid from the left and in red this grid shifted into $f_3$. Also shown on the right is the relation between an element $b$ on the red grid and its inverse. In this sense, $f_3$ is the origin of the red grid.} \label{fig:honeycomb2}
\end{figure}
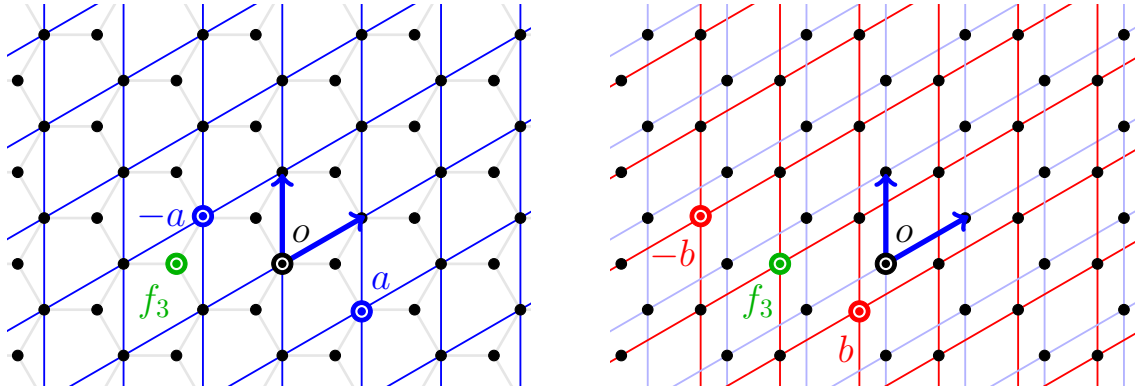

The inverse of an element $a=\alpha_1f_1+\alpha_2f_2$ on the blue grid is found by usual reflection about the origin $o$: $-a=(-\alpha_1)f_1+(-\alpha_2)f_2$. But the inverse of $b=\beta_1f_1+\beta_2f_2+f_3$ on the red grid is
\[
-b\ =\ -\beta_1f_1-\beta_2f_2-f_3\ =\ (-\beta_1)f_1+(-\beta_2)f_2+f_3,
\]
which is again sitting on the red grid and is given by reflection of $b$ about $f_3$. 

Note that the point $b$ shown in the right part of Figure~\ref{fig:honeycomb2} is exactly the point $h$ from Figure~\ref{fig:honeycomb} that made us wonder about the geometric relation between $h, -h$ and $o$.
\end{example}

The shift operator $V_g$ with $g=\delta_1f_1+\delta_2f_2+\delta_3 (3f_3)$ on $\ell^2(G)$ would then move data by $(\delta_1,\delta_2)$ hexagons in $(f_1,f_2)$-coordinates and rotate its location inside that target hexagon by $\delta_3 f_3$.

\section{Graphs and graph laplacians}
\label{sec:graph_laplacians}

In this section, we review results on graphs and laplacians on discrete graphs, see e.g.\ \cite{KellerLenzWojciechowskiBook}.
Let $G$ be a countable set which we equip with the discrete topology.
Following \cite[Defn.~1.1]{KellerLenzWojciechowskiBook}, $(b,c)$ is called a {\sl graph} over $G$ if $b: G\times G\to [0,\infty)$ is a symmetric function satisfying $b(k,k) = 0$ and $\sum_{\ell\in G} b(k,\ell) < \infty$ for all $k\in G$, and $c:G\to [0,\infty)$.

\textbf{Combinatorial graph distance and balls.}
Two points $k,\ell\in G$ are connected if $b(k,\ell)>0$; in this case we say that there is an edge between $k$ and $\ell$ and write $k\sim \ell$. The \emph{combinatorial graph distance}\footnote{Note that the actual value of a nonzero $b(k,\ell)$ is irrelevant for $d(k,\ell)$, which is just equal to $1$ if $b(k,\ell)\ne 0$.} $d:G\times G\to [0,\infty]$ is then given by 
\[d(k,\ell):= \inf\{n\in\N_0:\; \exists\, i_0,\ldots, i_n\in G: k=i_0\sim i_1\sim\ldots\sim i_n=\ell\},\]
where $\inf\varnothing=\infty$. The graph $(b,c)$ is connected if $d(k,\ell)<\infty$ for all $k,\ell\in G$.
For $k\in G$ and $r\geq 0$ we let $B(k,r):=\{\ell\in G:\; d(k,\ell)\leq r\}$ be the ball around $k$ of radius $r$. 

We say that a graph $(b,0)$ over $G$ has \emph{subexponential growth} if
\[\liminf_{r\to \infty} \inf_{k\in G} \frac{1}{r}\log \frac{\sharp B(k,r)}{\sharp B(k,1)} = 0.\]

\textbf{Energy form and Laplacian.}
A graph $(b,c)$ over $G$ comes with an energy form $\mathcal{Q}:=\mathcal{Q}_{b,c}: \mathcal{D}\times \mathcal{D}\to \R$ with
\begin{align*}
    \mathcal{D} &\ :=\ \Bigl\{x\in C(G):\; \frac{1}{2} \sum_{k,\ell\in G} b(k,\ell)|x_k-x_\ell|^2 + \sum_{k\in G} c(k) |x_k|^2 < \infty\Bigr\},\\
    \mathcal{Q}(x,y) &\ :=\ \frac{1}{2} \sum_{k,\ell\in G} b(k,\ell)(x_k-x_\ell)(y_k - y_\ell) + \sum_{k\in G} c(k) x_k y_k  \quad(x,y\in \mathcal{D}),
\end{align*}
where $C(G)$ is the space of all functions $x: G\to \R$. ($C(G)$ is also the space of continuous functions;  every function on $G$ is continuous as $G$ is discrete.) Moreover, the restriction of $\mathcal{Q}$ to
$D:=\overline{C_c(G)}^{\|\cdot\|_{\mathcal{Q}}}\cap \mathcal{D}$, where $C_c(G)$ is the space of compactly, i.e.\ finitely, supported functions and $\|\cdot\|_{\mathcal{Q}}^2 = \|\cdot\|_2^2 + \mathcal{Q}(\cdot,\cdot)$, yields a form $Q:=Q_{b,c}$ in $\ell^2(G)$ (again we equip $G$ with the counting measure) which provides an associated operator in $\ell^2(G)$, the {\sl Dirichlet Laplacian} $L:=L_{b,c}$ see \cite[p.~109]{KellerLenzWojciechowskiBook}. Note that $L$ is given by
\begin{align*}
    D(L) & = \{x\in D:\; \exists z\in\ell^2(G):\; \mathcal{Q}(x,y) = \langle z,y\rangle \quad(y\in D)\},\\ 
    L x_k & = \sum_{\ell\in G} b(k,\ell) (x_k-x_\ell) + c(k)x_k \quad(x\in D(L), k\in G).
\end{align*}

\textbf{Dirichlet Laplacian on subsets.}
Following \cite[\S1.3]{KellerLenzWojciechowskiBook}, for each finite $W\subseteq G$ we define $Q_W^{(D)}$ on $\ell^2(W)$ by $Q_W^{(D)}(x,y):=\mathcal{Q}(\iota_W x,\iota_W y)$, where $\iota_W: C(W) \to C(G)$ is the canonical embedding, i.e.\ the extension of $x$ by zero outside $W$. Then $Q_W^{(D)} = \mathcal{Q}_{b_W,c_W+d_W}$, where $b_W:=b|_{W\times W}$, $c_W:=c|_W$ and $d_W(k):=\sum_{\ell\in G\setminus W} b(k,\ell)$ for $k\in W$.
The non-negative self-adjoint operator $L_W^{(D)}$ associated with $Q_W^{(D)}$ is the {\sl Dirichlet Laplacian with respect to $W$}; see \cite[p.~118]{KellerLenzWojciechowskiBook}. We have
\begin{align*}
    D(L_W^{(D)}) & = \ell^2(W),\\ 
    L_W^{(D)} x_k & = \sum_{\ell\in W} b(k,\ell) (x_k-x_\ell) + \Bigl(c(k)+\sum_{\ell\in G\setminus W} b(k,\ell)\Bigr)x_k \quad(x\in \ell^2(W), k\in W).
\end{align*}

\section{Main results and proofs} \label{sec:main}

Throughout this section we assume that $(G,+)$ is a countable Abelian group and equip $G$ with the discrete topology. Let $Y$ be a Hilbert space and set $E:= \ell^2(G,Y)$ (i.e.\ we use the counting measure on $G$). We specialise to particular classes of Abelian groups in the next section.

The key steps are to show \eqref{eq:patchAx_intro} and to minimise the truncation penalty as a function of its parameters; the rest are fairly straightforward conclusions. We will do things step by step and for each of the methods $\tau_1$ and $\tau$. First some final preparations:

We start with a finite set $W\subset G$ that, when shifted around $G$, will serve as our truncation window. For example, in $G=\Z^d$ one may work with $W=\{1,\dots,n\}^d$ for $n\in\N$. 

\medskip\noindent
{\bf Convention.\ } To keep notations as simple as possible, let us hide the parameters $W$ and $n$ for most of the following computations. They will come back when we study the asymptotics of the truncation penalty as we blow up $W$ in an appropriate way.

With this convention, our patches $A_{n,k}$ and $x_{n,k}$ from the introduction become  $A_k$ and $\tilde x_k$, respectively\footnote{The tilde in this notation helps us to tell the vector patch $\tilde x_k$ from the vector entry $x_i$.}.

\medskip\noindent
{\bf The operator patches $A_k$.\ }
For $k\in G$, let $P_{k}$ denote the operator of multiplication by the characteristic function of the set $k+W$. Then the $k$th local patch $A_k$ of our band operator $A$ on $\ell^2(G,Y)$ is defined by
\begin{equation} \label{eq:patchA_tau1}
\ A_k^{(\tau_1)}\ :=\ AP_k
\end{equation}
in case of the $\tau_1$ method (one-sided truncation), and it is
\begin{equation} \label{eq:patchA_tau}
\ A_k^{(\tau)}\ :=\ P_kAP_k
\end{equation}
for the $\tau$ method (usual finite sections, two-sided truncation). Both operators are understood as acting on the 
space $\im P_k=\ell^2(k+W,Y)$. (Again, we may also use the symbol $P_k$ for $k\in G$ for the corresponding operator on $\ell^2(G)$.)


\medskip\noindent
{\bf The vector patches $\tilde x_k$: soft truncation and weights.\ }
Deriving the vector patches $\tilde x_k$ from $x$ by a ``sharp'' truncation (application of $P_k$, i.e.~multiplication by the characteristic function of $k+W$) will lead to a larger truncation penalty $\eps_n$ in \eqref{eq:patchAx_intro} than using a somewhat softer way to cut off: Take a nonzero weight function $w\in \ell^2(G)$ with $\supp w \subset W$ and compute the vector patches
\begin{equation} \label{eq:patchx}
\tilde x_k\ :=\ M_{V_kw}x,\quad\text{so that}\quad (\tilde x_k)_i\ =\ w_{i-k}\cdot x_i,\quad\text{i.e.}\quad (\tilde x_k)_{k+m}\ =\ w_m\cdot x_{k+m}
\end{equation}
for $k,i,m\in G$. 

The proof of \eqref{eq:patchAx_intro} needs to know $\sum_{k\in G} \|\tilde x_k\|^2$, which is the following simple computation:
\begin{equation}\label{eq:lem-trick1}
\sum_{k\in G} \|\tilde x_k\|^2\, =\, \sum_{k\in G}\|M_{V_kw}x\|^2\, =\, \sum_{k\in G}\sum_{i\in G}\|w_{i-k}x_i\|^2\, =\, \sum_{i\in G}\|x_i\|^2\sum_{k\in G}|w_{i-k}|^2\, =\, \|x\|_2^2\,\|w\|_2^2,
\end{equation}
where we use that $G$ is a group. (Note how \eqref{eq:lem-trick1} holds, e.g., for $G=\Z$ but not for $G=\N$.)

The case of the ``sharp'' truncation is still part of the consideration by putting $w=\chi_W$, the characteristic function of the set $W$, and we will later minimise the penalty $\eps_n$ from \eqref{eq:patchAx_intro} via variation of the weight function $w$. From our studies in \cite{CW.Heng.ML:JST} with $G=\Z$ we have learnt that the characteristic function $w=\chi_W$ (sharp cut-off) leads to $\eps_n\sim 1/\sqrt n$, while a hat-shaped or a $\sin$-shaped weight $w$ leads to $\eps_n\sim 1/n$, the latter proven to be optimal.

\subsection{The \texorpdfstring{$\tau_1$}{tau1} method: one-sided truncations}

In this subsection on the $\tau_1$ method we use the abbreviation $A_k := A_k^{(\tau_1)}$, for $k\in G$. 

\subsubsection{The $\tau_1$ version of \eqref{eq:patchAx_intro} with weight-dependent truncation penalty}
We first derive our $\tau_1$ version of \eqref{eq:patchAx_intro}, where the truncation penalty still depends on the choice of the weight function $w$, giving us potential for optimisation below.

\begin{proposition} \label{prop:tau1}
Let $A\in \W(E)$, in which case $A$ has the representation \eqref{eq:AW},
for some sequence $(b^{(j)})_{j\in G} \subset \ell^\infty(G,L(Y))$ with $\|A\|_\W <\infty$, and 
suppose a finite set $W\subset G$ and  
 $w\in \ell^2(G)$ with $w\neq 0$ and $\supp w\subset W$ 
 are given. Then, for every nonzero $x\in E$, there exists a $k\in G$ such that $\tilde x_k\ne 0$ and
\begin{equation} \label{eq:patchAx_tau1}
\frac{\|A_{k}\, \tilde x_{k}\|}{\| \tilde x_{k}\|}\ \le\ \frac{\|Ax\|}{\|x\|}\ +\ \eps(w,A),
\end{equation}
where
\begin{equation} \label{eq:pen_tau1}
\eps(w,A)\ :=\ \eps^{(\tau_1)}(w,A)\ :=\ \sum_{j\in G} \|b^{(j)}\|_\infty \frac{\|V_jw-w\|_2}{\|w\|_2}
\end{equation}
and where the local patches $A_k$ and $\tilde x_k$ are defined by \eqref{eq:patchA_tau1} and \eqref{eq:patchx} above.
\end{proposition}
\begin{proof}
Take a nonzero $x\in E$. Recalling the commutator notation $[A,B]:=AB-BA$ for bounded operators $A,B$, we have, for every $k\in G$,
\begin{eqnarray*}
A_k\tilde x_k &=& AP_kM_{V_kw}x \ =\ AM_{V_kw}x \ =\ M_{V_kw}Ax\ +\ [A,M_{V_kw}]x\\
&=& M_{V_kw}Ax\ +\ \left[\sum_{j\in G} M_{b^{(j)}}V_j\ ,\ M_{V_kw}\right]x
\ =\ M_{V_kw}Ax\ +\ \sum_{j\in G} [M_{b^{(j)}}V_j,M_{V_kw}]x,
\end{eqnarray*}
so that
\[
\|A_k\tilde x_k\| \ \le\ \|M_{V_kw}Ax\|\ +\ \sum_{j\in G} \|[M_{b^{(j)}}V_j,M_{V_kw}]x\|.
\]
Applying squares, summation over all $k\in G$ and a square root, it follows that
\begin{eqnarray}
\sqrt{\sum_{k\in G} \|A_k\tilde x_k\|^2} &\le&
\sqrt{\sum_{k\in G}\left(\underbrace{\|M_{V_kw}Ax\|}_{=:\ \xi_k}\ +\ \sum_{j\in G} \underbrace{\|[M_{b^{(j)}}V_j,M_{V_kw}]x\|}_{=:\ \eta_{j,k}}\right)^2}\nonumber\\ 
&=& \sqrt{\sum_{k\in G}(\xi_k+\sum_{j\in G} \eta_{j,k})^2}\ 
\ \stackrel{(M)}\le
\ \sqrt{\sum_{k\in G} \xi_k^2}+\sum_{j\in J}\sqrt{\sum_{k\in G}\eta_{j,k}^2}\nonumber\\
&=& \sqrt{\sum_{k\in G}\|M_{V_kw}Ax\|^2}\ +\ \sum_{j\in G}\sqrt{\sum_{k\in G}\|[M_{b^{(j)}}V_j,M_{V_kw}]x\|^2} \label{eq:tag1}
\end{eqnarray}
with $(M)$ denoting an application of Minkowski's inequality.

Now we evaluate the right-hand side of \eqref{eq:tag1}. First, by \eqref{eq:lem-trick1},
\begin{equation} \label{eq:trick1}
\sqrt{\sum_{k\in G}\|M_{V_kw}Ax\|^2}=\sqrt{\|w\|_2^2\|Ax\|^2}=\|w\|_2\|Ax\|.
\end{equation}
%
Further, for $j,k\in G$,
\begin{eqnarray*}
 \|[M_{b^{(j)}}V_j,M_{V_kw}]x\|^2 &=& \sum_{i\in G}|(M_{b^{(j)}}V_jM_{V_kw}x)_i-(M_{V_kw}M_{b^{(j)}}V_jx)_i|^2\\
&=& \sum_{i\in G}\|b^{(j)}_i \cdot w_{i-j-k}\cdot x_{i-j}-w_{i-k}\cdot b^{(j)}_i\cdot x_{i-j}\|^2\\
&\le& \|b^{(j)}\|_\infty^2\sum_{i\in G} |w_{i-j-k}-w_{i-k}|^2\|x_{i-j}\|^2.
\end{eqnarray*}
Substituting\footnote{Note that we use commutativity of the group operation when we turn $i-j-k$ into $i-k-j=m-j$.} $m:=i-k$, summing up over $k\in G$, and arguing as in \eqref{eq:lem-trick1},
\begin{eqnarray}
\sum_{k\in G} \|[M_{b^{(j)}}V_j,M_{V_kw}]x\|^2 &\le& \|b^{(j)}\|_\infty^2 \sum_{k\in G} \sum_{m\in G} |w_{m-j}-w_{m}|^2\|x_{m+k-j}\|^2 \nonumber\\
&=& \|b^{(j)}\|_\infty^2 \sum_{m\in G} |w_{m-j}-w_{m}|^2\sum_{k\in G} \|x_{m+k-j}\|^2\nonumber\\
&=& \|b^{(j)}\|_\infty^2 \|V_jw-w\|_2^2\|x\|_2^2.\label{eq:RHS2}
\end{eqnarray}
  
So from \eqref{eq:tag1} and our bounds \eqref{eq:trick1} and \eqref{eq:RHS2} on its right-hand side we conclude
\begin{eqnarray}
\sqrt{\sum_{k\in G}\|A_k\tilde x_k\|^2} &\le& \|w\|_2\|Ax\| + \sum_{j\in J} \|b^{(j)}\|_\infty \|V_jw-w\|_2\|x\|\nonumber\\
&=& \underbrace{\left(\frac{\|Ax\|}{\|x\|} + \sum_{j\in J} \|b^{(j)}\|_\infty \frac{\|V_jw-w\|_2}{\|w\|_2}\right)}_{=:\,c}\|w\|_2\,\|x\|.\label{eq:tag2}
\end{eqnarray}
By \eqref{eq:lem-trick1} again, 
\[
\sum_{k\in G}\|\tilde x_k||^2\ =\ \sum_{k\in G}\|M_{V_kw}x||^2\ =\ \|w\|_2^2\,\|x\|^2,
\]
so that \eqref{eq:tag2} turns into
\[
\sqrt{\sum_{k\in G}\|A_k\tilde x_k\|^2} \ \le\ c\sqrt{\sum_{k\in G}\|\tilde x_k||^2}
\qquad\text{and hence}\qquad
\sum_{k\in G}\|A_k\tilde x_k\|^2 \ \le\ c^2\sum_{k\in G}\|\tilde x_k||^2.
\]
We conclude that, with $c$ as defined in \eqref{eq:tag2}, either i) $\|A_k\tilde x_k\|^2\le c^2\|\tilde x_k\|^2$ for all $k\in G$, where $\tilde x_k\ne 0$ for at least one of them or ii) $\|A_k\tilde x_k\|^2> c^2\|\tilde x_k\|^2$ for some $k$ but $\|A_k\tilde x_k\|^2< c^2\|\tilde x_k\|^2$ for others. Also then $\tilde x_k\ne 0$ for the latter. It remains to take the square root and divide by $\|\tilde x_k\|$.
\end{proof}

\begin{remark}
We can see that the truncation penalty $\eps(w,A)$ comes entirely from the commutator term $[A,M_{V_kw}]x$ in the first line of the argument. If $G=\Z^d$ and if our weight function $w$ comes from $\Z^d$-evaluations of a compactly supported continuous function $\ph:\R^d\to\R$, whose argument is getting stetched out, meaning that $w=\ph_t|_{\Z^d}$, where $\ph_t(x)=\ph(\frac xt)$ for $t>1$, then, by \cite[Thm~1.42]{LiBook}, we have $\|[A,M_{V_kw}]\|\to 0$ as $t\to\infty$. This is a first hint towards $\eps(w,A)\to 0$ as $t\to\infty$ if $\ph$ is continuous with compact support and $w=\ph_t|_{\Z^d}$.
\end{remark}

\subsubsection{Minimising the truncation penalty by variation of the weight function $w$} \label{sec:min_tau1}

We now seek to find a computable bound for the minimal value of $\varepsilon(w,A)$, given by \eqref{eq:pen_tau1}, where we minimise over the set of all nonzero $w\in \ell^2(G)$ supported in the finite set $W$.

In the following proposition we use the notation $\W_p(E)$ defined in \eqref{eq:Wpdef}, and recall that we define $0^p:= 0$ for $p=0$. The assumption $A\in \W_p(E)\subset \W(E)$ in this proposition is crucial and guarantees: 
(i) that we can apply Proposition \ref{prop:tau1} to bound $\eps(w,A)$; 
(ii) that $\sum_{j\in G} (\|b^{(j)}\|_\infty +\|b^{(-j)}\|_\infty)^{2-p}$ is finite; 
(iii) that the graph $(b,c)$ we define in the proposition is well defined, in particular that $\sum_{\ell\in G} b(k,\ell) < \infty$ for all $k\in G$. 
Note also that $\eps_p(W,A)$, as defined in this proposition, depends on the operator $A$, the choice of the window $W$, and also on the choice of $p$; if $A\in \W_q(E)$, with $0\leq q<1$, then $A\in \W_p(E)$ for $q\leq p\leq 1$ so \eqref{eq:Weps} holds for $q\leq p\leq 1$. As we note in Remark \ref{rem:p01} below, the statement of this proposition simplifies in the important special cases $p=0$ and $1$, when $\W_p(E) = BO(E)$ and $\W(E)$, respectively.

\begin{proposition}
\label{prop:min_eps_tau1}
Suppose, for some $0\leq p\leq 1$,  that $A\in \W_p(E)\subset \W(E)$, in which case $A$ has the representation \eqref{eq:AW},
for some sequence $(b^{(j)})_{j\in G} \subset \ell^\infty(G,L(Y))$ with $w_p(A) <\infty$, and let $W\subset G$ finite. Let $(b,c)$ be the graph over $G$ defined by
\[
b(k,\ell) := 2\Bigl(\frac{\|b^{(k-\ell)}\|_\infty + \|b^{(l-k)}\|_\infty}{2}\Bigr)^p \chi_{G\setminus\{0\}} (k-l)
\quad (k,\ell\in G),
\] 
and $c(k):= 0$ for $k\in G$.
Then
\begin{align}
\nonumber
\inf_{\begin{array}{c}{\scriptstyle w\in\ell^2(G)\setminus\{0\},}\\{\scriptstyle \supp w\subset W}\end{array}}\varepsilon(w,A) 
&\ \leq\ \sqrt{\sum_{j\in G\setminus\{0\}} \Bigl(\frac{\|b^{(j)}\|_\infty + \|b^{(-j)}\|_\infty}{2}\Bigr)^{2-p}} \sqrt{\min \Spec L_W^{(D)}}\\
\label{eq:Weps}
&\ =:\ \eps_p^{(\tau_1)}(W,A)\ =:\ \eps_p(W,A),
\end{align}
where $L_W^{(D)}$ is the Dirichlet Laplacian with respect to $W$ for $(b,c)$, so that $\min \Spec L_W^{(D)}>0$.
\end{proposition}

\begin{proof} For $j\in G$ we define $B_j:= (V_j-I)P_0$. Let $w\in \ell^2(G)$ with $w\neq 0$ and $\supp w\subset W$. For $j\in G$ the shift $V_{-j}$ is isometric; hence, $\|B_jw\|_2 = \|V_{-j}B_j w\|_2 = \|B_{-j}w\|_2$.
By \eqref{eq:pen_tau1}, we thus obtain with an application of the Cauchy--Schwarz inequality,
\begin{align*}
\eps(w,A)\  & =\ \sum_{j\in G\setminus\{0\}} \|b^{(j)}\|_\infty \frac{\|B_j w\|_2}{\|w\|_2}
 =\ \sum_{j\in G\setminus\{0\}} \Bigl(\frac{\|b^{(j)}\|_\infty + \|b^{(-j)}\|_\infty}{2}\Bigr)^{1-p/2+p/2} \frac{\|B_j w\|_2}{\|w\|_2} \\
 & \leq\ \sqrt{\sum_{j\in G\setminus\{0\}} \Bigl(\frac{\|b^{(j)}\|_\infty + \|b^{(-j)}\|_\infty}{2}\Bigr)^{2-p}} \sqrt{\sum_{j\in G\setminus\{0\}}  \Bigl(\frac{\|b^{(j)}\|_\infty + \|b^{(-j)}\|_\infty}{2}\Bigr)^p\frac{\|B_jw\|_2^2}{\|w\|_2^2}}.
\end{align*}
Now, for $j\in G\setminus\{0\}$,
$$
\|B_jw\|_2^2\ =\ \langle B_j^HB_j w,w\rangle,
$$
where $B_j^H = P_0(V_{-j}-I)$ is the Hilbert space adjoint of $B_j$. Hence,
\[\sum_{j\in G\setminus\{0\}}  \Bigl(\frac{\|b^{(j)}\|_\infty + \|b^{(-j)}\|_\infty}{2}\Bigr)^p\|B_jw\|_2^2 = \Bigl\langle \!\sum_{j\in G\setminus\{0\}}  \Bigl(\frac{\|b^{(j)}\|_\infty + \|b^{(-j)}\|_\infty}{2}\Bigr)^p B_j^H B_jw,w \Bigr\rangle = \langle Bw, w\rangle,
\]
where
\begin{align*}\
B &\ :=\ \sum_{j\in G\setminus\{0\}} \Bigl(\frac{\|b^{(j)}\|_\infty + \|b^{(-j)}\|_\infty}{2}\Bigr)^p B_j^HB_j\\
&\ =\ P_0\left(\sum_{j\in G\setminus\{0\}} \Bigl(\frac{\|b^{(j)}\|_\infty + \|b^{(-j)}\|_\infty}{2}\Bigr)^p(2I-V_j-V_{-j})\right)P_0.
\end{align*}
Thus
$$
\eps(w,A)\ \leq \ \sqrt{\sum_{j\in G\setminus\{0\}} \Bigl(\frac{\|b^{(j)}\|_\infty + \|b^{(-j)}\|_\infty}{2}\Bigr)^{2-p}}  \sqrt{\frac{\langle Bw, w\rangle}{\|w\|_2^2}}.
$$
Further, a short calculation reveals that $B$ restricted to $\im(P_0)=\ell^2(W)$ coincides with $L_W^{(D)}$. As 
\[\frac{\langle Bw,w\rangle}{\|w\|_2^2}\]
is the Rayleigh quotient of $B=L_W^{(D)}$ for the vector $w$, we obtain that it is minimised by the ground state $w\neq 0$ of $L_{W}^{(D)}$, with its minimal value being $\min \Spec L_{W}^{(D)}>0$, the minimal eigenvalue.
\end{proof}

\begin{remark} \label{rem:p01}
Let us comment on the special situations $p=0$ and $p=1$.
\begin{enumerate}
 \item[(a)]
 For $p=0$ we have $\W_0(E)=BO(E)$, and the above proposition simplifies, with $J:=\{j\in G:\; \|b^{(j)}\|_\infty + \|b^{(-j)}\|_\infty \neq 0\}$,
$$
\eps_0(W,A)\ =\ \frac{1}{2}\sqrt{\sum_{j\in J\setminus\{0\}} \Bigl(\|b^{(j)}\|_\infty + \|b^{(-j)}\|_\infty\Bigr)^2} \sqrt{\min \Spec L_W^{(D)}}.
$$
\item[(b)] 
For $p=1$ we have $\W_p(E)=\W(E)$, and the above proposition simplifies to
$$
\eps_1(W,A)\ =\ \sqrt{\sum_{j\in G\setminus\{0\}} \|b^{(j)}\|_\infty} \sqrt{\min \Spec L_W^{(D)}}.
$$
\end{enumerate}
\end{remark}


\subsubsection{Upper and lower bounds on the lower norm and the pseudospectrum} \label{sec:bounds_tau1}

\begin{corollary}[\bf Lower $\tau_1$ bound on the lower norm (weight-dependent)] \label{cor:tau1}
Under the assumptions and with the notations of Proposition \ref{prop:tau1}, it holds that
\[
\inf_{k\in G}\nu(A_k)\ \le\ \nu(A)\ +\ \eps(w,A),
\]
where we recall that $A_k=A_k^{(\tau_1)}=AP_k$ acts from $\im P_k=\ell^2(k+W,Y)$ to $\ell^2(G,Y)$ for each $k\in G$.
\end{corollary}
\begin{proof}
For all $x\in \ell^2(G,Y)\setminus\{0\}$ and the corresponding $k\in G$ in \eqref{eq:patchAx_tau1}, we have
\[
\nu(A_k)\ \le\ \frac{\|A_{k}\, \tilde x_{k}\|}{\| \tilde x_{k}\|}\ \le\ \frac{\|Ax\|}{\|x\|}\ +\ \eps(w,A).
\]
Now ignore the middle term and pass to $\inf_{k\in G}$ and to $\inf_{x\in \ell^2(G,Y)\setminus\{0\}}$, to arrive at our claim.
\end{proof}

As a special feature of the $\tau_1$ method we also have this very natural upper bound on $\nu(A)$:

\begin{lemma}[\bf Upper $\tau_1$ bound on the lower norm] \label{lem:upper_tau1}
For every bounded linear operator $A$ on $\ell^2(G,Y)$ and every $k\in G$ it holds that
\[
\nu(A)\ \le\ \nu(A_k).
\]
\end{lemma}
\begin{proof}
By $A_k=(AP_k)|_{\im P_k}=A|_{\im P_k}=A|_{\ell^2(k+W,Y)}$, we have $\nu(A_k)=\nu_{k+W}(A)\ge\nu_G(A)=\nu(A)$ for all $k\in G$, by \eqref{eq:numon}.
\end{proof}

So we can complement the lower bound from Corollary \ref{cor:tau1} as follows:
\begin{corollary}[\bf Upper and lower $\tau_1$ bounds on the lower norm (weight-dependent)] \label{cor:tau1_sandwich_w}
Under the assumptions and with the notations of Proposition \ref{prop:tau1}, it holds that
\[
\inf_{k\in G}\nu(A_k)\ -\ \eps(w,A)\ \le\ \nu(A)\ \le\ \inf_{k\in G}\nu(A_k).
\]
\end{corollary}

\begin{remark}
If we rearrange the start of the proof of Proposition \ref{prop:tau1} into
\[
\|M_{V_kw}Ax\|\ \le\ \|AM_{V_kw}x\| \ +\ \|[M_{V_kw},A]x\|\ \le\ \|A_k\tilde x_k\| \ +\ \sum_{j\in J}\|[M_{V_kw},M_{b^{(j)}}V_j]x\|,
\]
apply squares, summation over $k\in G$, a square root, Minkowski's inequality, our bounds \eqref{eq:RHS2} on $\sum_k\|[M_{V_kw},M_{b^{(j)}}V_j]x\|^2$ and \eqref{eq:lem-trick1}, we derive the existence of a $k\in G$ with
\[
\frac{\|A_{k}\, \tilde x_{k}\|}{\| \tilde x_{k}\|}\ \ge\ \frac{\|Ax\|}{\|x\|}\ -\ \eps(w,A).
\]
If we then carry on in a manner symmetric to the proof of Corollary \ref{cor:tau1_sandwich_w}, we get
\[
\sup_{k\in G}\|A_k\|\ +\ \eps(w,A)\ \ge\ \|A\|\ \ge\ \sup_{k\in G}\|A_k\|
\]
with the very same penalty term $\eps(w,A)$ as in \eqref{eq:pen_tau1}.
Summarising, our techniques do not only yield $\eps(w,A)$-good approximations of $\nu(A)$ but also of $\|A\|$ in terms of (lower) norms of operator patches $A_k$. This fact is not a property of a particular operator $A$ but of the underlying space $\ell^2(G,Y)$, see \cite[Prop.~3.4]{HagLiSei} and the discussion following its proof there.
\end{remark}

Making use of the minimisation result from Section \ref{sec:min_tau1}, we can replace the weight-dependent penalty term $\eps(w,A)$ by the window-dependent penalty term $\eps_p(W,A)$ from \eqref{eq:Weps}:

\begin{corollary}[\bf Upper and lower $\tau_1$ bounds on the lower norm (window-dependent)] \label{cor:tau1_sandwich_W}
Let $0\leq p\leq 1$, $A\in \W_p(E)$ as in \eqref{eq:A} on $\ell^2(G,Y)$ with a countable Abelian group $(G,+)$ and a Hilbert space $Y$, and
let a finite truncation window $W\subset G$ be given. Then it holds that
\begin{equation} \label{eq:sandw_nuA}
\inf_{k\in G}\nu(A_k)\ -\ \eps_p(W,A)\ \le\ \nu(A)\ \le\ \inf_{k\in G}\nu(A_k),
\end{equation}
where the operator patches $A_k$ are defined by \eqref{eq:patchA_tau1} and $\eps_p(W,A)$ is the window-dependent penalty term defined in \eqref{eq:Weps}.
\end{corollary}
\begin{proof}
Note that Corollary \ref{cor:tau1_sandwich_w} holds for every $w\in\ell^2(G,Y)\setminus\{0\}$ with $\supp w\subset W$ and hence with $\eps(w,A)$ replaced by its corresponding infimum. Then use inequality \eqref{eq:Weps}.
\end{proof}

Applying Corollary \ref{cor:tau1_sandwich_W} to $A^*$, noting that $\eps_p(W,A^*)=\eps_p(W,A)$ by \eqref{eq:Weps} since $A$ and $A^*$ have the same diagonal suprema, we arrive at
\begin{equation} \label{eq:sandw_nuA*}
\inf_{k\in G}\nu((A^*)_k)\ -\ \eps_p(W,A)\ \le\ \nu(A^*)\ \le\ \inf_{k\in G}\nu((A^*)_k).
\end{equation}
Passing to the minimum of \eqref{eq:sandw_nuA} and \eqref{eq:sandw_nuA*}, we get
\begin{equation} \label{eq:sandw_nuAA*}
\inf_{k\in G}\min\{\nu(A_k),\nu((A^*)_k)\}\ -\ \eps_p(W,A)\ \le\ \underbrace{\min\{\nu(A),\nu(A^*)\}}_{\mu(A)}\ \le\ \inf_{k\in G}\min\{\nu(A_k,\nu((A^*)_k)\}.
\end{equation}
From here we can easily jump to the conclusion for the pseudospectrum, where we recall that, for bounded operators $A:\ell^2(G,Y)\to\ell^2(G,Y)$, we have $\Specn_0A=\Spec A$.

\begin{corollary}[\bf Upper and lower $\tau_1$ inclusions for the pseudospectrum] \label{cor:speps_tau1}
Let $0\leq p\leq 1$, $A\in \W_p(E)$ as in \eqref{eq:A} on $\ell^2(G,Y)$ with a countable Abelian group $(G,+)$ and a Hilbert space $Y$, and
let a finite truncation window $W\subset G$ be given. 

{\bf a) } Then it holds that
\[
\begin{array}{rcll}
\gamma_{\eps}^W(A)\ \subset&\!\speps A\! &\subset\  \gamma_{\eps+\eps_p(W,A)}^W(A),& \quad\eps>0\qquad\text{and}\\[1mm]
\Gamma_{\eps}^W(A)\ \subset&\!\Speps A\! &\subset\  \Gamma_{\eps+\eps_p(W,A)}^W(A),& \quad\eps\ge 0,
\end{array}
\]
where $\eps_p(W,A)$ is the window-dependent penalty term defined in \eqref{eq:Weps}, the operator patches $A_k$ and $(A^*)_k$ are defined by \eqref{eq:patchA_tau1} and
\begin{eqnarray*}
\gamma_{\eps}^W(A)&:=&\Big\{\lambda\in\C: \inf_{k\in G}\min\big\{\nu(A_k-\lambda I_k),\nu((A^*)_k-\lambda I_k)\big\}<\eps\Big\},\quad \eps> 0\qquad\text{as well as}\\
\Gamma_{\eps}^W(A)&:=&\Big\{\lambda\in\C: \inf_{k\in G}\min\big\{\nu(A_k-\lambda I_k),\nu((A^*)_k-\lambda I_k)\big\}\le\eps\Big\},\quad \eps\ge 0.
\end{eqnarray*}
{\bf b) } If $\dim Y<\infty$ then one has
\[
\gamma_{\eps}^W(A)\ =\ \bigcup_{k\in G} \big(\speps A_k\cup \speps (A^*)_k\big),\qquad\eps>0,
\]
where the pseudospectra of $A_k$ and $(A^*)_k$ are to be understood in the sense of \eqref{eq:speps.rect}, while $\Gamma_{\eps}^W(A)$ is in general neither $\bigcup_{k\in G} \big(\Speps A_k\cup \Speps (A^*)_k\big)$ nor its closure.
\end{corollary}
\begin{proof}
{\bf a) } Replace $A$ by $A-\lambda I$ in \eqref{eq:sandw_nuAA*}, noting that $\eps_p(W,A-\lambda I)=\eps_p(W,A)$ for $\lambda\in\C$, by \eqref{eq:Weps}, and $(\lambda I)^*=\lambda I$, not $\overline\lambda I$. Now assume the right-hand side of that inequality is $<\eps$, resp.~$\le\eps$.

{\bf b) } The statement about $\gamma_{\eps}^W(A)$ is clear since $\inf_k y_k<\eps$ if and only if $y_k<\eps$ for some $k$. For the statement about $\Gamma_{\eps}^W(A)$, see Section 5.2 of \cite{CW.Heng.ML:JST}, in particular the examples there.
\end{proof}

The difficulties with $\Gamma_{\eps}^W(A)$ in statement b) are connected with the non-square operator patches; they will resolve for the square operator patches in the $\tau$ and $\pi$ methods.

Note that the sets $\gamma_{\eps}^W(A)$ and $\Gamma_{\eps}^W(A)$ depend on $W$ via the operator patches $A_k$, where we suppress $W$ in the notation. Now this $W$-dependence takes center stage:

\subsubsection{Convergence of the pseudospectral inclusions}
We show that if 
$(W_n)_{n\in\N}$ is a sequence of finite subsets of $G$ such that $W_n\subseteq W_{n+1}$ for all $n\in\N$ and $\bigcup_{n\in\N} W_n = G$
then both set sandwiches of Corollary \ref{cor:speps_tau1} a) converge to the pseudospectrum that they bound from above and below, $\speps A$, resp.~$\Speps A$. The convergence of a sequence of sets that we have in mind is to be understood as follows:

The Hausdorff distance of two bounded sets $S,T\subset\C$ is
\[
d_H(S,T)\ :=\ \max\Big(\sup_{s\in S}\dist(s,T),\sup_{t\in T}\dist(t,S) \Big),
\quad\text{where}\quad
\dist(s,T)=\inf_{t\in T}|s-t|.
\]
It is a metric on the set of all compact subsets of $\C$ and a pseudometric on the set of all bounded subsets because $d_H(S,T)=0$ as soon as $\clos S=\clos T$. For bounded sets $S,S_1,S_2,\dots\subset\C$ we say that $S_n$ Hausdorff-converges to $S$ and write $S_n\to S$ if $d_H(S_n,S)\to 0$, noting that in this context of a pseudometric, the statement $S_n\to S$ only determines the closure of the limit $S$, not $S$ itself.

By \eqref{eq:closspeps}, for $\eps>0$, a set sequence converges to $\speps A$ if and only if it converges to $\Speps A$.

As a useful alternative characterisation (see e.g. \cite[Prop.~3.6]{HaRoSi2}) of Hausdorff convergence for bounded sequences $(S_n)$ of bounded sets in $\C$, one has that
\begin{equation} \label{eq:Haus_liminf_limsup}
S_n\to S\qquad\iff\qquad
\liminf_{n\to\infty} S_n\ =\ \limsup_{n\to\infty}S_n\ =\ S,
\end{equation}
where $\liminf S_n$ is the set of all limits of sequences $(s_n)$ with $s_n\in S_n$ for all $n\in\N$, and $\limsup S_n$ is the set of all partial limits (i.e., accumulation points) of such sequences $(s_n)$. 
Note that 
\begin{equation} \label{eq:liminf_clos}
\liminf \clos S_n\ =\ \liminf S_n\ =\ \clos\liminf S_n
\end{equation}
holds as well as the same formula with $\limsup$.

Since $G$ is a group, a graph $(b,c)$ over $G$ is connected if and only if $B(o,1)$ is a generating set of $G$, where $o$ is the unit element of $G$. In particular, the graph $(b,c)$ from Proposition \ref{prop:min_eps_tau1} is connected if and only if $\{j\in G:\; \|b^{(j)}\|_\infty + \|b^{(-j)}\|_\infty \neq 0\}$ is a generating set.

\begin{corollary}[\bf Convergence of the $\tau_1$ inclusion bounds to the pseudospectrum] \label{cor:conv_tau1}
Let $0\leq p\leq 1$, $A\in\W_p(E)$.
Assume that the graph $(b,c)$ over $G$ from Proposition \ref{prop:min_eps_tau1} has subsexponential growth and is connected.
If $\eps>0$ and 
$(W_n)_{n\in\N}$ is a sequence of finite subsets of $G$ such that $W_n\subseteq W_{n+1}$ for all $n\in\N$ and $\bigcup_{n\in\N} W_n = G$
then both the upper and lower inclusion bounds from Corollary \ref{cor:speps_tau1} a) Hausdorff-converge as $n\to\infty$ to the set that they bound, i.e.~to the pseudospectrum of $A$. If $\eps=0$ then $\Gamma_{\eps_p(W_n,A)}^{W_n}\to\Specn_0 A=\Spec A$.
\end{corollary}
\begin{proof}
Let $L$ denote the Dirichlet Laplacian on $\ell^2(G)$ with $b$ and $c$ as in Proposition \ref{prop:min_eps_tau1}. By subexponential growth and \cite[Theorem 13.8]{KellerLenzWojciechowskiBook}, one has $\min\Spec L = 0$. By \cite[Lemma 1.21]{KellerLenzWojciechowskiBook} we observe that $L_{W_n}^{(D)}\to L$ in (generalised) strong resolvent sense. By self-adjointness, we have that every $\lambda\in \Spec L$ (including $\lambda=0$) is the limit of a sequence $(\lambda_n)$ with $\lambda_n\in \Spec L_{W_n}^{(D)}$, see e.g. \cite[Thm.~VIII.24(a)]{ReedSimon1} or an obvious modification of the argument in \cite[Thm.~7.2]{HaRoSi2}. Consequently, $\min\Spec L_{W_n}^{(D)}\to 0$ as $n\to\infty$, whence also $\eps_p(W_n,A)\to 0$, by \eqref{eq:Weps}.

Let $\eps>0$ and let 
$n\in\N$ be large enough that $\eps_p(W_n,A)\in (0,\eps)$. 
By Corollary~\ref{cor:speps_tau1}~a), 
\begin{equation} \label{eq:5}
\specn_{\eps-\eps_p(W_n,A)}A\ \subset \gamma_{\eps}^{W_n}(A)\ \subset\ \speps A\ \subset\  \gamma_{\eps+\eps_p(W_n,A)}^{W_n}(A)\ \subset\ \specn_{\eps+\eps_p(W_n,A)}A.
\end{equation}
By sending 
$n\to \infty$, all sets in \eqref{eq:5} converge to $\speps A$ since $\eps_p(W_n,A)\to 0$ and by Hausdorff-continuity of the map $\eps\mapsto\speps A$.  The same arguments show that $\Gamma_{\eps}^{W_n}(A)\to\Speps A$ and $\Gamma_{\eps+\eps_p(W_n,A)}^{W_n}(A)\to\Speps A$ if $\eps>0$. For $\eps=0$, use $\Specn_0 A \subset  \Gamma_{0+\eps_p(W_n,A)}^{W_n}(A) \subset \Specn_{0+\eps_p(W_n,A)}A$.
\end{proof}

\begin{remark}
    Let $W\subseteq \widetilde{W}\subseteq G$ be finite. Then $\varepsilon_p(W,A)\geq \varepsilon_p(\widetilde{W},A)$. Indeed, let $w\in \ell^2(W)$ be a normalized eigenfunction for $L_{W}^{(D)}$ to the eigenvalue $\min \Spec L_{W}^{(D)}$ and let $\widetilde{w}$ be the extension of $w$ by zero to $\widetilde{W}$. Then
    \begin{align*}
        \min \Spec L_{W}^{(D)}
        & = Q_W^{(D)}(w,w) = \mathcal{Q}(\iota_W w, \iota_W w) = \mathcal{Q}(\iota_{\widetilde{W}} \widetilde{w}, \iota_{\widetilde{W}} \widetilde{w}) = Q_{\widetilde{W}}^{(D)} (\widetilde{w},\widetilde{w})\\
        & \geq \inf_{x\in \ell^2(\widetilde{W}), \|x\|_2=1} Q_{\widetilde{W}}^{(D)}(x,x) = \min \Spec L_{\widetilde{W}}^{(D)},
    \end{align*}
    which yields the assertion.
\end{remark}

\subsection{The \texorpdfstring{$\tau$}{tau} method: two-sided truncations}
For $k\in G$ abbreviate $P_kAP_k=A_k^{(\tau)}=:A_k$ here since this subsection is all about the $\tau$ method. Note that this time we have $A_k:\im P_k=\ell^2(k+W,Y)\to\im P_k=\ell^2(k+W,Y)$.

\subsubsection{The $\tau$ version of \eqref{eq:patchAx_intro} with weight-dependent truncation penalty}
\begin{proposition} \label{prop:tau}
Let $A\in \W(E)$, in which case $A$ has the representation \eqref{eq:AW},
for some sequence $(b^{(j)})_{j\in G} \subset \ell^\infty(G,L(Y))$ with $\|A\|_\W <\infty$, and 
suppose a finite set $W\subset G$ and
$w\in \ell^2(G)$ with $w\neq 0$ and $\supp w\subset W$ 
 are given. Then, for every nonzero $x\in E$, there exists a $k\in G$ such that $\tilde x_k\ne 0$ and
\begin{equation} \label{eq:patchAx_tau}
\frac{\|A_{k}\, \tilde x_{k}\|}{\| \tilde x_{k}\|}\ \le\ \frac{\|Ax\|}{\|x\|}\ +\ \eps(w,A),
\end{equation}
where
\begin{equation} \label{eq:pen_tau}
\eps(w,A)\ :=\ \eps^{(\tau)}(w,A)\ :=\ \sum_{j\in G} \|b^{(j)}\|_\infty \frac{\|P_0V_jw-w\|_2}{\|w\|_2}
\end{equation}
and where the patches $A_k$ and $\tilde x_k$ are defined by \eqref{eq:patchA_tau} and \eqref{eq:patchx} above.
\end{proposition}
\begin{proof}
The proof is almost identical to that of Proposition \ref{prop:tau1}. We skip steps that are literally the same and we take the liberty of colouring terms in \textcolor{blue}{blue} that are introduced or changed through the extra truncation $P_k$ after the action of $A$.

Take a nonzero $x\in E$. We have, for every $k\in G$,
\begin{eqnarray*}
A_k\tilde x_k &=& \textcolor{blue}{P_k}AP_kM_{V_kw}x\ =\ \textcolor{blue}{P_k}AM_{V_kw}x \ =\ \textcolor{blue}{P_k}M_{V_kw}Ax\ +\ \textcolor{blue}{P_k}[A,M_{V_kw}]x\\
&=& M_{V_kw}Ax\ +\ \textcolor{blue}{P_k}[A,M_{V_kw}]x
\ =\ M_{V_kw}Ax\ +\ \sum_{j\in G} \textcolor{blue}{P_k}[M_{b^{(j)}}V_j,M_{V_kw}]x,
\end{eqnarray*}
so that
\[
\|A_k\tilde x_k\| \ \le\ \|M_{V_kw}Ax\|\ +\ \sum_{j\in G} \|\textcolor{blue}{P_k}[M_{b^{(j)}}V_j,M_{V_kw}]x\|.
\]
Applying squares, summation $\sum_{k\in G}$, a square root, and Minkowski's inequality, it follows that 
\begin{eqnarray}
\sqrt{\sum_{k\in G} \|A_k\tilde x_k\|^2} &\le&
 \sqrt{\sum_{k\in G}\|M_{V_kw}Ax\|^2}\ +\ \sum_{j\in J}\sqrt{\sum_{k\in G}\|\textcolor{blue}{P_k}[M_{b^{(j)}}V_j,M_{V_kw}]x\|^2} \label{eq:tag1'}.
\end{eqnarray}
Now again we evaluate the right-hand side of \eqref{eq:tag1'} one by one: First, \eqref{eq:trick1} holds unchanged.
Then we bound $\|\textcolor{blue}{P_k}[M_{b^{(j)}}V_j,M_{V_kw}]x\|^2$ for $j,k\in G$ as follows: 
\begin{eqnarray*}
 \|\textcolor{blue}{P_k}[M_{b^{(j)}}V_j,M_{V_kw}]x\|^2 &=& \sum_{i\in \textcolor{blue}{k+W}}\|(M_{b^{(j)}}V_jM_{V_kw}x)_i-(M_{V_kw}M_{b^{(j)}}V_jx)_i\|^2\\
&=& \sum_{i\in \textcolor{blue}{k+W}}\|b^{(j)}_i\cdot  w_{i-j-k}\cdot x_{i-j}-w_{i-k}\cdot b^{(j)}_i\cdot x_{i-j}\|^2\\
&\le& \|b^{(j)}\|_\infty^2\sum_{i\in \textcolor{blue}{k+W}} |w_{i-j-k}-w_{i-k}|^2\|x_{i-j}\|^2.
\end{eqnarray*}
Substituting $i-k=:m\in \textcolor{blue}{W}$, summing up over $k\in G$, and arguing as in \eqref{eq:lem-trick1},
\begin{eqnarray}
\sum_{k\in G} \|\textcolor{blue}{P_k}[M_{b^{(j)}}V_j,M_{V_kw}]x\|^2 &\le& \|b^{(j)}\|_\infty^2 \sum_{k\in G} \sum_{m\in \textcolor{blue}{W}} |w_{m-j}-w_{m}|^2\|x_{m+k-j}\|^2 \nonumber\\
&=& \|b^{(j)}\|_\infty^2 \sum_{m\in \textcolor{blue}{W}} |w_{m-j}-w_{m}|^2\sum_{k\in G} \|x_{m+k-j}\|^2\nonumber\\
&=& \|b^{(j)}\|_\infty^2 \|\textcolor{blue}{P_0}V_jw-w\|_2^2\|x\|_2^2.\label{eq:RHS2'}
\end{eqnarray}
From \eqref{eq:tag1'} and the bounds \eqref{eq:trick1} and \eqref{eq:RHS2'} we continue as in the proof of Proposition \ref{prop:tau1}.
\end{proof}

\subsubsection{Minimising the truncation penalty by variation of the weight function $w$}
\label{sec:min_tau}

We again optimise $\varepsilon(w,A)$ over nonzero $w\in \ell^2(G)$ supported in $W$.

\begin{proposition}
\label{prop:min_eps_tau}
Suppose, for some $0\leq p\leq 1$,  that $A\in \W_p(E)\subset \W(E)$, in which case $A$ has the representation \eqref{eq:AW},
for some sequence $(b^{(j)})_{j\in G} \subset \ell^\infty(G,L(Y))$ with $w_p(A) <\infty$, and let $W\subset G$ finite. Let $(b,c)$ be the graph over $G$ defined by
\[
b(k,\ell) := \blue{b^W(k,\ell)} := 2\Bigl(\frac{\|b^{(k-\ell)}\|_\infty + \|b^{(l-k)}\|_\infty}{2}\Bigr)^p \chi_{G\setminus\{0\}} (k-l)\blue{\chi_W(k)\chi_W(\ell)}
\quad (k,\ell\in G),
\] 
and $c(k):= \blue{c^W(k):= 2\sum_{j\in G\setminus\{0\}, k-j\in G\setminus W} \Bigl(\frac{\|b^{(k-j)}\|_\infty + \|b^{(j-k)}\|_\infty}{2}\Bigr)^p}$ for $k\in G$.
Then
\begin{align} \nonumber
\inf_{\begin{array}{c}{\scriptstyle w\in\ell^2(G)\setminus\{0\},}\\{\scriptstyle \supp w\subset W}\end{array}}\varepsilon(w,A) 
&\ \leq\ \sqrt{\sum_{j\in G\setminus\{0\}} \Bigl(\frac{\|b^{(j)}\|_\infty + \|b^{(-j)}\|_\infty}{2}\Bigr)^p} \sqrt{\min \Spec L_W^{(D)}}\\
\label{eq:Weps_tau}
&\ =:\ \eps_p^{(\tau)}(W,A)\ =:\ \eps_p(W,A),
\end{align}
where $L_W^{(D)}$ is the Dirichlet Laplacian with respect to $W$ for $(b,c)$, so that $\min \Spec L_W^{(D)}>0$.
\end{proposition}

\begin{proof} For $j\in G$ we define $B_j:= \blue{P_0}(V_j-I)P_0$. Let $w\in \ell^2(G)$ with $w\neq 0$ and $\supp w\subset W$. Reasoning as in the proof of Proposition \ref{prop:min_eps_tau1}, we obtain
\begin{align*}
\eps(w,A)\ \leq \ \sqrt{\sum_{j\in G\setminus\{0\}} \Bigl(\frac{\|b^{(j)}\|_\infty + \|b^{(-j)}\|_\infty}{2}\Bigr)^{2-p}}  \sqrt{\frac{\langle Bw, w\rangle}{\|w\|_2^2}},
\end{align*}
where
\[B:= \sum_{j\in G\setminus\{0\}} \Bigl(\frac{\|b^{(j)}\|_\infty + \|b^{(-j)}\|_\infty}{2}\Bigr)^p B_j^HB_j.\]

Further, a short calculation reveals that $B$ restricted to $\im(P_0)=\ell^2(W)$ coincides with $L_W^{(D)}$, which yields the assertion.
\end{proof}

Although \eqref{eq:Weps_tau} literally coincides with \eqref{eq:Weps} for the $\tau_1$ method, the value is different since the graph $(b,c)$ and hence the Laplacian $L_W^{(D)}$ is different.

\begin{remark} \label{rem:p01_tau}
Let us again comment on the special situations $p=0$ and $p=1$.
\begin{enumerate}
 \item[(a)]
 For $p=0$ we observe, with $J:=\{j\in G:\; \|b^{(j)}\|_\infty + \|b^{(-j)}\|_\infty \neq 0\}$,
$$
\eps_0(W,A)\ =\ \frac{1}{2}\sqrt{\sum_{j\in J\setminus\{0\}} \Bigl(\|b^{(j)}\|_\infty + \|b^{(-j)}\|_\infty\Bigr)^2} \sqrt{\min \Spec L_W^{(D)}}.
$$
\item[(b)] 
For $p=1$ we obtain
$$
\eps_1(W,A)\ =\ \sqrt{\sum_{j\in G\setminus\{0\}} \|b^{(j)}\|_\infty} \sqrt{\min \Spec L_W^{(D)}}.
$$
\end{enumerate}
\end{remark}

\subsubsection{Bounds on the lower norm and the pseudospectrum} \label{sec:bounds_tau}
We proceed as in Subsection \ref{sec:bounds_tau1} but with some small changes: Firstly, we do not have the upper bound on $\nu(A)$ and hence not the lower bounds on $\speps A$ and $\Speps A$. And secondly, on the bright side, we need not mention $(A^*)_k$ in our pseudospectral bounds since
\begin{equation} \label{eq:Ak*}
(A^*)_k\ =\ P_kA^*P_k\ =\ P_k^*A^*P_k^*\ =\ (P_kAP_k)^*\ =\ (A_k)^*
\end{equation}
and by the following elementary lemma:
%

\begin{proof}
$A$ and $A^*$ are either a) both singular or b) both invertible. In case a), by standard linear algebra arguments for finite square matrices (note: $\dim Y<\infty$), both have a nontrivial kernel, so that $\nu(A)=0=\nu(A^*)$. In case b), $\nu(A)=\nu(A^*)>0$ by e.g. \cite[Lemma 2.10]{HagLiSei}.
\end{proof}

\begin{corollary}[\bf Lower bound on the lower norm] \label{cor:nu_tau}
Let $A\in \W(E)$ as in \eqref{eq:A} on $\ell^2(G,Y)$ with a countable Abelian group $(G,+)$ and a Hilbert space $Y$,
let a finite truncation window $W\subset G$ and a weight function $w\in\ell^2(G)$ with $w\neq 0$ and $\supp w\subset W$ be given. Then it holds that
\[
\inf_{k\in G}\nu(A_k)\ \le\ \nu(A)\ +\ \eps(w,A),
\]
where the operator patches $A_k$ are defined by \eqref{eq:patchA_tau}. 
Moreover, if $0\leq p\leq 1$ and $A\in \W_p(E)$, then
\[\inf_{k\in G}\nu(A_k)\ \le\ \nu(A)\ +\ \eps_p(W,A),\]
where $\eps_p(W,A)$ is the window-dependent penalty term defined in \eqref{eq:Weps_tau}. 
\end{corollary}
\begin{proof}
For all $x\in \ell^2(G)\setminus\{0\}$ and the corresponding $k\in G$ in \eqref{eq:patchAx_tau}, we have
\[
\nu(A_k)\ \le\ \frac{\|A_{k}\, \tilde x_{k}\|}{\| \tilde x_{k}\|}\ \le\ \frac{\|Ax\|}{\|x\|}\ +\ \eps(w,A).
\]
By passing to $\inf_{k\in G}$ and to $\inf_{x\in \ell^2(G)\setminus\{0\}}$, we conclude the first inequality. The second inequality follows by passing to $\inf_{w\ne 0, \supp w\subset W}$ in the first, followed by application of \eqref{eq:Weps_tau}.
\end{proof}
Again, apply Corollary \ref{cor:nu_tau} to $A$ and $A^*$ and take the minimum of the two window-dependent inequalities, noting that, again, $\eps_p(W,A^*)=\eps_p(W,A)$, by \eqref{eq:Weps_tau}. Doing so, we get
\begin{equation} \label{eq:nuAA*_tau}
\inf_{k\in G}\underbrace{\min\{\nu(A_k),\nu((A^*)_k)\}}_{=\,\mu(A_k)}\ -\ \eps_p(W,A)\ \le\ \min\{\nu(A),\nu(A^*)\}\ =\ \mu(A),
\end{equation}
where the underbrace equality holds by \eqref{eq:Ak*}. Again replace $A$ by $A-\lambda I$ and assume the right-hand side is $<\eps$, resp.~$\le\eps$:
\begin{corollary}[\bf Inclusions for the pseudospectrum] \label{cor:speps_tau}
Let $0\leq p\leq 1$, $A\in \W_p(E)$ as in \eqref{eq:A} on $\ell^2(G,Y)$ with a countable Abelian group $(G,+)$ and a Hilbert space $Y$, and
let a finite truncation window $W\subset G$ be given. Then it holds that
\begin{eqnarray}
\label{eq:tau-incl1}
\speps A&\subset& \qquad\bigcup\limits_{k\in G} \specn_{\eps+\eps_p(W,A)} A_k,\qquad \eps>0\qquad\text{and}\\[0mm]
\label{eq:tau-incl2}
\Speps A& \subset&  \clos\bigcup\limits_{k\in G} \Specn_{\eps+\eps_p(W,A)} A_k,\qquad \eps\ge 0,
\end{eqnarray}
where $\eps_p(W,A)$ is the window-dependent penalty term defined in \eqref{eq:Weps_tau} and the operator patches $A_k$ are defined by \eqref{eq:patchA_tau} and where $\Specn_0A=\Spec A$.
\end{corollary}
Note how, in contrast to Corollary \ref{cor:speps_tau1} for the $\tau_1$ method, here
\begin{equation} \label{eq:3.1_tau}
\big\{\lambda\in\C: \inf_{k\in G}\mu(A_k-\lambda I_k)\le\eps\big\}\ =\ \clos\bigcup_{k\in G} \Specn_{\eps+\eps_p(W,A)} A_k,
\end{equation}
see  the argument in (the proof of) \cite[Proposition 3.1]{CW.Heng.ML:JST}.
Also note that $\Specn_{\eps+\eps_p(W,A)} A_k$ can be replaced by $\specn_{\eps+\eps_p(W,A)} A_k$ in \eqref{eq:tau-incl2} since $\clos\!\cup_k \clos S_k=\clos\!\cup_k S_k$ for any family of sets $S_k$. With this modification, \eqref{eq:tau-incl2} simply follows from \eqref{eq:tau-incl1} by taking the closure on both sides. The important case $\eps=0$, however, is new in \eqref{eq:tau-incl2}.

To study the potential convergence of the $\tau$ enclosures from Corollary \ref{cor:speps_tau} to the pseudospectrum of $A$, again take a sequence $W_1, W_2,\dots$ of finite subsets of $G$ such that $W_n\subseteq W_{n+1}$ for all $n\in\N$ and $\bigcup_{n\in\N} W_n = G$. Because of our focus on asymptotics with respect to $n$, let us briefly bring $n$ back to the notation of the operator patches:
\begin{equation} \label{eq:Akn}
A_{k,n}\ :=\ P_{k,n}AP_{k,n},
\quad\text{where}\quad
P_{k,n}x\ :=\ \chi_{k+W_n}\cdot x,
\qquad k\in G,\ n\in\N.
\end{equation}
To compensate for the absence of Lemma \ref{lem:upper_tau1} in the $\tau$ method, we make one more assumption here:

We say that the $\tau$ patches $(A_{k,n})_{k\in G, n\in\N}$ of $A$ {\sl do not suffer from spectral pollution} if
\begin{equation} \label{eq:nopoll}
\forall\eps>0:\quad \limsup_{n\to\infty} \bigcup_{k\in G} \Speps A_{k,n}\ \subset\ \Speps A.
\end{equation}
Under this additional assumption, we have convergence also of the $\tau$ enclosure sets to the pseudospectrum:

\begin{corollary}[\bf Convergence of the $\tau$ inclusions to the pseudospectrum] \label{cor:conv_tau}
Let $0\leq p\leq 1$, $A\in\W_p(E)$.
Assume that the graph $(b^G,c^G)$ over $G$ from Proposition \ref{prop:min_eps_tau} has subsexponential growth and is connected.
Let $\eps>0$ and 
$(W_n)_{n\in\N}$ be a sequence of finite subsets of $G$ such that $W_n\subseteq W_{n+1}$ for all $n\in\N$ and $\bigcup_{n\in\N} W_n = G$. If the $\tau$ patches $(A_{k,n})_{k\in G, n\in\N}$ from \eqref{eq:Akn} do not suffer from spectral pollution
then the upper inclusion bounds from Corollary \ref{cor:speps_tau} Hausdorff-converge as $n\to\infty$ to $\Speps A$.
\end{corollary}

Since our graph parameters $b=b^W$ and $c=c^W$ from Proposition \ref{prop:min_eps_tau} depend on the choice of $W$, we have to argue a bit for the (generalised) strong resolvent convergence as $n\to \infty$. Apart from this, the proof is very similar to the proof of Corollary \ref{cor:conv_tau1}.

\begin{proof}
We start by showing that $\eps_p(W_n,A)\to 0$ as $n\to\infty$.
For $n\in\N$ note that $b^{W_n} = b^G_{W_n}$ on $W_n\times W_n$, and $c^{W_n}(k)\to 0$ for all $k\in G$ as well as $\|c^{W_n}(\cdot)\|_\infty \leq 2w_p(A)$. This yields $M_{c^{W_n}}\to 0$ strongly and also $M_{c^{W_n}_{W_n}}\to 0$ strongly.
Moreover, $L_{W_n}^{(D)} = L_{b^G_{W_n},c^{W_n}_{W_n}} = L_{b^G_{W_n},0} + M_{c^{W_n}_{W_n}}$. Note that $L_{b^G_{W_n},0} = L_{W_n}^{(D)}$ for $c=0$, and hence,
by \cite[Lemma 1.21]{KellerLenzWojciechowskiBook}, for $\alpha>0$ we observe $(L_{b^G_{W_n},0}+\alpha)^{-1}\to (L_{b^G,0}^{(D)}+\alpha)^{-1}$ strongly. By the second resolvent identity we have
\[(L_{W_n}^{(D)}+\alpha)^{-1} - (L_{b^G,0}^{(D)}+\alpha)^{-1} = (L_{W_n}^{(D)}+\alpha)^{-1}\bigl(L_{b^G,0}^{(D)} - L_{b^G_{W_n},0} - M_{c^{W_n}_{W_n}}\bigr) (L_{b^G,0}^{(D)}+\alpha)^{-1}.\]
Note that for $f\in \ell^2(G)$ and $k\in G$ we obtain
\[L_{b^G,0}^{(D)} f_k - L_{b^G_{W_n,0}} f_k = \begin{cases} \sum_{\ell\in G\setminus W_n} b^g(k,\ell)(f_k-f_\ell) & k\in W_n,\\
L_{b^G,0}^{(D)} f_k & k\notin W_n.
\end{cases}\]
Since $\sum_{\ell\in G\setminus W_n}b^G(k,\ell)|f_k-f_\ell| \leq 2\sum_{\ell\in G\setminus W_n} \left(\frac{\|b^{(k-\ell)}\|_\infty + \|b^{(\ell-k)}\|_\infty}{2}\right)^p 2\|f\|_\infty \to 0$ (as $w_p(A)<\infty$) and $\sum_{k\in G\setminus W_n} |L_{b^G,0}^{(D)} f_k|^2 \to 0$ (as $L_{b^G,0}^{(D)} f\in \ell^2(G)$), we obtain
\[L_{b^G,0}^{(D)} f - L_{b^G_{W_n,0}} f \to 0\]
in $\ell^2(G)$.
This then yields $(L_{W_n}^{(D)}+\alpha)^{-1} \to (L_{b^G,0}^{(D)}+\alpha)^{-1}$ strongly.
Now, since $\min\Spec L_{b^G,0}^{(D)} = 0$ by \cite[Theorem 13.8]{KellerLenzWojciechowskiBook}, we again obtain $\min\Spec L_{W_n}^{(D)}\to 0$ and hence $\eps_p(W_n,A)\to 0$ as $n\to\infty$, arguing as in the proof of Corollary \ref{cor:conv_tau1}.

Now let $\eps>0$ and fix $m\in\N$. By $\eps_p(W_n,A)\to 0$,  we have 
\begin{equation} \label{eq:eps<1m}
\eps_p(W_n,A)\ \le\ \frac1m
\qquad\text{for all sufficiently large } n.
\end{equation}
Then we conclude that
\begin{align*}
\Speps A &=\ \liminf_{n\to\infty} \Speps A\ \stackrel{\eqref{eq:tau-incl2}}\subset\ \liminf_{n\to\infty}\clos\bigcup\limits_{k\in G} \Specn_{\eps+\eps_p(W_n,A)} A_{k,n}\\
&\stackrel{\eqref{eq:liminf_clos}}=\ \liminf_{n\to\infty}\bigcup\limits_{k\in G} \Specn_{\eps+\eps_p(W_n,A)} A_{k,n}\ \subset\ \limsup_{n\to\infty} \bigcup\limits_{k\in G} \Specn_{\eps+\eps_p(W_n,A)} A_{k,n}\\ 
&\stackrel{\eqref{eq:eps<1m}}\subset\ \limsup_{n\to\infty} \bigcup\limits_{k\in G} \Specn_{\eps+\frac1m} A_{k,n}
\ \stackrel{\eqref{eq:nopoll}}\subset\ \Specn_{\eps+\frac1m}A
\end{align*}
for all $m\in\N$. By the continuity of $\eps\mapsto \Speps A$, all sets in this chain of inclusions Hausdorff-converge to $\Speps A$ as $m\to\infty$;
in particular, abbreviating $S_n:=\cup_k \Specn_{\eps+\eps_p(W_n,A)} A_{k,n}$, we have $\Speps A=\liminf_n S_n=\limsup_n S_n$, so that $S_n\to \Speps A$ as $n\to\infty$, by \eqref{eq:Haus_liminf_limsup}.
\end{proof}

\section{Examples}
\label{sec:examples}

In this section we will discuss a few examples.

\subsection{Comparison with our 1D result from \texorpdfstring{\cite{CW.Heng.ML:JST}}{\ref{bib-CW.Heng.ML:JST}}}
In the setting and notations of \cite{CW.Heng.ML:JST} we have $G=\Z$, $b^{(j)} \neq 0$ if and only if $J=\{-1,0,1\}$ and $(b^{(1)},b^{(0)},b^{(-1)})=(\alpha,\beta,\gamma)$.

To evaluate the penalty term \eqref{eq:pen_tau1} for the $\tau_1$ method in this 1D situation, note that $V_0=I$ and that, with $B:=(V_{-1}-I)|_{\im P_0}$ and $C:=(V_1-I)|_{\im P_0}$, we have $B^*B=C^*C$, so that both $B$ and $C$ have the same smallest singular value, say $s$, and the same corresponding singular vector, say $w$. So the expression for the minimum of \eqref{eq:pen_tau1} simplifies to $(\|\alpha\|_\infty+\|\gamma\|_\infty)s$, where $s=\frac{\|Bw\|}{\|w\|}=\frac{\|Cw\|}{\|w\|}$, which is exactly the minimum of the penalty term (5.2) in \cite[Proposition 5.1]{CW.Heng.ML:JST}.

The penalty term \eqref{eq:pen_tau} for the $\tau$ method, again since $V_0=I$, in this 1D situation coincides with the penalty term (3.9) in \cite[Proposition 3.3]{CW.Heng.ML:JST} since $T_n^-=\|P_0V_1w-w\|_2^2$, $T_n^+=\|P_0V_{-1}w-w\|_2^2$ and $S_n=\|w\|_2^2$.

\subsection{The case \texorpdfstring{$G=\Z^d$}{G=Zd}}\label{sec:Zd}

One possible generalisation of \cite{CW.Heng.ML:JST}, important for applications, is the multidimensional case, i.e.\ $G:=\Z^d$, which we will focus in this subsection.

We start with a special case of band operators.
Let $J:=\{k\in \Z^d:\; \|k\|_1\leq 1\} = \{0\}\cup\{\pm e_1,\ldots,\pm e_d\}$ for the standard unit vectors $e_l = (\delta_{l,m})_{m\in \{1\,\ldots,d\}}$, and let $p=0$ and $A\in\W_0(\ell^2(G))$ (i.e.\ $A$ is a band matrix) with $b^{(j)} \neq 0$ only if $j\in J$. In particular, for $d=1$ the infinite matrix $A$ is tridiagonal.

Moreover, let $W_n:=\{1,\ldots,n\}^d$ for $n\in\N$. Then it is easy to see that $L_{W_n}^{(D)}$ in Proposition \ref{prop:min_eps_tau1} agrees with two times the standard finite difference discrete Laplacian for the standard Dirichlet Laplacian on $[0,n+1]^d$ with interior grid points given by $W_n$.
Then, as is well known, the minimal eigenvalue is given by
\[\min \Spec L_{W_n}^{(D)} = 8d \sin\Bigl(\frac{\pi}{2(n+1)}\Bigr)^2,\]
so, by Remark \ref{rem:p01},
\begin{align*}
    \varepsilon_0(W_n,A) & = \frac{1}{2}\sqrt{\sum_{j\in J\setminus\{0\}} \left(\|b^{(j)}\|_\infty+\|b^{(-j)}\|_\infty\right)^2} \cdot 2\sqrt{2d}\sin\Bigl(\frac{\pi}{2(n+1)}\Bigr) \sim \frac{1}{n}
\end{align*}
for large $n$. Note that this bound is realised by the choice $w=\big(\prod_{i=1}^d\sin(\frac{l_i \pi}{n+1})\big)_{(l_1,\ldots,l_d)\in W_n}$, which is the corresponding eigenfunction of $L_{W_n}^{(D)}$.

For the choice of a sharp cutoff $w:=\chi_W$, we can obtain
\[\frac{\langle L_{W_n}^{(D)}w, w\rangle}{\|w\|_2^2} = \frac{2(n^d-(n-2)^d)}{n^d} \sim \frac{4dn^{d-1}}{n^d} = \frac{4d}{n},\]
and thus
\[\varepsilon(w,A) \sim \frac{1}{2}\sqrt{\sum_{j\in J\setminus\{0\}} \left(\|b^{(j)}\|_\infty+\|b^{(-j)}\|_\infty\right)^2} \cdot 2\frac{\sqrt{d}}{\sqrt{n}},\]
so the decay in $n$ is much slower in this case.

Specialising to the case $d=1$, we essentially recover the results in \cite[Corollary 5.4]{CW.Heng.ML:JST} for $\varepsilon_0(W,A)$ apart from the additional factor $\sqrt{2d} = \sqrt{2}$ and the $\ell^2$-sum of the norms of the diagonals $(b^{(j)})_j$ rather than the $\ell^1$-sum (which comes from the application of the Cauchy--Schwarz inequality). This can be seen as a consequence of the estimate of the $\ell^1$-sum of the diagonals by their $\ell^2$-sum, which yields the additional factor $\sqrt{2}$.

Moreover, also the corresponding result in \cite[Proposition 5.1]{CW.Heng.ML:JST} for $\varepsilon(\chi_W,A)$ can be recovered up to analogous adjustments.

Now, let $A\in \W_0(E)$ be a band operator. By introducing a suitable block structure with block sizes corresponding to the band width, we can reduce the general case of a band operator to the special case above; see \cite[Section 1.4]{CW.Heng.ML:JST} for the case $d=1$. In particular, $\eps_0(W_n,A)\sim \frac{1}{n}$ for all band operators $A\in \W_0(E)$.

\subsection{Finitely generated groups}

Let $G$ be a finitely generated Abelian group. Then, according to \eqref{eq:fundthm}, $G \cong \Z^d\oplus H$, where $H$ is a finite sum of cyclic groups $\Z/q\Z$ with $q\in \N$. Then, we observe 
\[
\ell^2(G)\ \cong\ \ell^2(\Z^d)\otimes \ell^2(H) \ \cong\ \ell^2(\Z^d,\ell^2(H))
\] 
and hence, for a Hilbert space $Y$, 
\[
\ell^2(G,Y)\ \cong\ \ell^2(G)\otimes Y\ \cong\ \ell^2(\Z^d) \otimes \ell^2(H)\otimes Y\ \cong\ \ell^2(\Z^d, \ell^2(H)\otimes Y),
\]
so this case actually behaves as the case in Subsection \ref{sec:Zd}.

In particular, if $A\in \W_0(E)$ is a band operator, then $A$ is also band operator considered as operator on $\ell^2(\Z^d, \ell^2(H)\otimes Y)$, so also in this case we obtain the $\frac{1}{n}$ rate for the window-dependent truncation error.

\subsection{On minimising the window-dependent truncation error w.r.t.\ \texorpdfstring{$p$}{p}}

In \eqref{eq:Weps} in Proposition \ref{prop:min_eps_tau1} as well as in \eqref{eq:Weps_tau} in Proposition \ref{prop:min_eps_tau} the window-dependent truncation error $\eps_p(W,A)$ depends on the choie of $p\in [0,1]$. We now show that in different examples different values for $p$ may be optimal.

\begin{example}\label{ex:Z2p=1}
    Let $G:=\Z^2$ and $A\in \W_0(\ell^2(\Z^2))$ as in \eqref{eq:A} with $b^{(j)}\neq 0$ if and only if $j=(\pm 1,0)$ or $j=(0,\pm 1)$. Let us assume that $\|b^{(1,0)}\|_\infty + \|b^{(-1,0)}\|_\infty=2$ and $\|b^{(0,1)}\|_\infty + \|b^{(0,-1)}\|_\infty=2\alpha>0$. Moreover, as in Subsection \ref{sec:Zd}, let $W:=\{1,\ldots,n\}^2$ for some $n\in\N$. Then, for the $\tau_1$ method and $0\leq p\leq 1$ we observe, by \eqref{eq:Weps},
    \[\eps_p(W,A) = \sqrt{2(1+\alpha^{2-p})} \sqrt{\min \Spec L_W^{(D)}},\]
    where, as in Proposition \ref{prop:min_eps_tau1}, 
    \[L_W^{(D)} = P_{(0,0)}\left(2(2I-V_{(1,0)} - V_{(-1,0)}) + 2\alpha^p (2I-V_{(0,1)}-V_{(0,-1)})\right)P_{(0,0)}.\]
    Thus, for $\alpha\neq 1$, this is an anisotropic finite difference Laplacian on $[0,n+1]^2$ with Dirichlet boundary conditions, and hence
    \[\min\Spec L_W^{(D)} = 4\sin\left(\frac{\pi}{2(n+1)}\right)^2 (2+2\alpha^p).\]
    Hence,
    \[\eps_p(W,A) = 4 \sqrt{1+\alpha^{2-p}} \sqrt{1+\alpha^p} \sin\left(\frac{\pi}{2(n+1)}\right)=
    4\sqrt{1+\alpha^2+\alpha(\beta_p+\beta_p^{-1})} \sin\left(\frac{\pi}{2(n+1)}\right)\]
    with $\beta_p:=\alpha^{1-p}$. Since $t+t^{-1}\ge 2$ for any $t>0$ with equality iff $t=1$, the minimal value of $\eps_p(W,A)$ w.r.t.\ $p$ is attained for $\beta_p=1$, i.e.~for $p=1$ (since $\alpha\ne 1$), that is,\
    \[\min_{p\in[0,1]}\eps_p(W,A)\ =\ \eps_1(W,A)\ =\ 4(1+\alpha) \sin\left(\frac{\pi}{2(n+1)}\right).\]
\end{example}

We now adjust Example \ref{ex:Z2p=1} in order to see that the optimal value for $p$ may depend on the particular operator $A$ and that the optimal $p$ may indeed be an interior point of $[0,1]$. 
\begin{example}
    Again, let $G:=\Z^2$, and let $A\in \W_0(\ell^2(\Z^2))$ as in \eqref{eq:A} with $b^{(j)}\neq 0$ if and only if $j=(\pm 1,0)$ or $j=(0,\pm 2)$. Let $\|b^{(1,0)}\|_\infty + \|b^{(-1,0)}\|_\infty=2\alpha_1>0$ and $\|b^{(0,2)}\|_\infty + \|b^{(0,-2)}\|_\infty=2\alpha_2>0$. Moreover, similarly as in Subsection \ref{sec:Zd}, let $W:=\{1,\ldots,n\}^2$ for some $n\in\N$. Then, for the $\tau_1$ method and $0\leq p\leq 1$, we observe, by \eqref{eq:Weps},
    \[\eps_p(W,A) = \sqrt{2(\alpha_1^{2-p} + \alpha_2^{2-p}))} \sqrt{\min \Spec L_W^{(D)}},\]
    where, as in Proposition \ref{prop:min_eps_tau1}, 
    \[L_W^{(D)} = P_{(0,0)}\left(2\alpha_1^p(2I-V_{(1,0)} - V_{(-1,0)}) + 2\alpha_2^p (2I-V_{(0,2)}-V_{(0,-2)})\right)P_{(0,0)}.\]
    Thus, $L_W^{(D)}$ is a tensor product of the two one-dimensional operators
    \begin{align*}
        L^j & := P_0 (2\alpha_j^p (2I-V_j-V_{-j}))P_0, \qquad j\in\{1,2\}.
    \end{align*}
    For $j\in\{1,2\}$, let $w^j\in \ell^2(\{1,\ldots,n\})$ be the ground state for $L^j$ and $\lambda_j$ the ground state energy.
    Then 
    \[\min \Spec L_W^{(D)} = \lambda_1+\lambda_2,\]
    and the ground state is given by $w:=w^1\otimes w^2$.
    We now fine-tune parameters. In order to do this, for $j\in\{1,2\}$, let 
    \[\widetilde{L}^j:=P_0 (2(2I-V_j-V_{-j}))P_0,\]
    such that $L^j = \alpha_j^p \widetilde{L}^j$; hence the ground state is still $w_j$ but the ground state energy is now $\widetilde{\lambda}_j = \frac{1}{\alpha_j^p} \lambda_j$.
    We now let $k>0$ and set $\alpha_j:= \widetilde{\lambda}_j^k$ for $j\in\{1,2\}$.
    Then 
    \begin{align*}
    \eps_p(W,A) &\ =\ \sqrt{2} \sqrt{\widetilde{\lambda}_1^{2k+1} + \widetilde{\lambda}_2^{2k+1} +\widetilde{\lambda}_1^{2k-kp} \widetilde{\lambda}_2^{1+kp} + \widetilde{\lambda}_1^{1+kp} \widetilde{\lambda}_2^{2k-kp}}\\
    &\ =\ \sqrt{2}\sqrt{\widetilde{\lambda}_1^{2k+1} + \widetilde{\lambda}_2^{2k+1} + (\widetilde{\lambda}_1\widetilde{\lambda}_2)^{\frac{2k+1}2}(\gamma_p+\gamma_p^{-1})},
    \end{align*}
    where $\gamma_p:=(\widetilde{\lambda}_1/\widetilde{\lambda}_2)^{k-\frac 12 - kp}$.
    Hence, the minimal value w.r.t. $p$ is attained for $\gamma_p=1$, that is when $k-\frac 12-kp=0$, i.e.~$p=\frac{2k-1}{2k}$. Via the choice of $k>\frac 12$, this optimal $p$ could be any interior point of $[0,1]$. Summarising,
    \[
    \min_{p\in[0,1]}\eps_p(W,A)\ =\ \eps_{\frac{2k-1}{2k}}(W,A)\ =\ \sqrt{2} \left(\widetilde{\lambda}_1^{\frac{2k+1}{2}} + \widetilde{\lambda}_2^{\frac{2k+1}{2}}\right).
    \]
    \end{example}

\section{Directions of extension} \label{sec:extensions}
To keep the exposition as simple as possible up to this point, we have refrained from incorporating the following technical extensions so far.

\subsection{Banach space-valued \texorpdfstring{$\ell^2$}{l2} spaces}
\label{subsec:Banach space-valued}
Instead of $Y$ being a Hilbert space, we can consider that case that $Y$ is a complex Banach space.

One peculiarity when passing from Hilbert space valued to Banach space valued sequences is that some of the level sets of the resolvent norm function $\lambda\mapsto \|(A-\lambda I)^{-1}\|$ of a bounded linear operator $A$ on $\ell^2(G,Y)$ can turn from curves in the complex plane into sets that contain a nonempty open set. In the latter case, the function $\eps\mapsto \speps A$ has a jump at such a level $\eps>0$ and equality \eqref{eq:closspeps} fails to hold. In these cases some of our arguments break down or need adapting. Here are the details:

We say that a complex Banach space $E$ has {\sl Globevnik's property} if $E$ is finite-dimensional or if $E$ is complex uniformly convex (see \cite[Definition 2.4(ii)]{Shargo08}) or if its dual space $E^*$ is complex uniformly convex. If $E$ has Globevnik's property then, for bounded linear operators $A$ on $E$, the level sets of the resolvent norm function are curves without nonempty open subsets by \cite{Globevnik74,Globevnik76,Shargo08,Shargo09}, so that the function $\eps\mapsto \speps A$ is Hausdorff-continuous and \eqref{eq:closspeps} holds; otherwise these properties might fail, see e.g.~\cite{Shargo08,Shargo09}. 

Hilbert spaces have Globevnik’s property (this is what we treated so far), and $\ell^2(G,Y)$ has it if $Y$ has it \cite{BoykoKadets}.

\begin{proposition} \label{prop:Globevnik}
Let $Y$ be a complex Banach space and $G$ countable Abelian group.
If $Y$ has Globevnik's property then all our arguments, equations and theorems extend to $\ell^2(G,Y)$. 
%
\end{proposition}



\subsection{\texorpdfstring{$\ell^p$}{lp} spaces with \texorpdfstring{$1<p<\infty$}{1<p<infty}}
Almost everything that has been said extends from $\ell^2(G,Y)$ to $\ell^p(G,Y)$ with $p\in (1,\infty)$. An exception is the minimisation of the truncation penalty as a function of the weight $w$ that has been tuned to the Hilbert space case $p=2$ here. 

The case $p=\infty$ is different in many respects, including of course the formula by which $\|x\|_p$ is defined. Because our arguments for $\|A^{-1}\|$ and pseudospectra involve the adjoint operator $A^*$, we also have to exclude $p=1$, noting that, however, our two versions of \eqref{eq:patchAx_intro} are still valid in $\ell^1(G,Y)$.

\subsection{Band-dominated operators} 
Our band operators \eqref{eq:A}, built by addition and composition from finitely many multiplication operators $M_b$ and shifts $V_j$, form an operator algebra in $L(\ell^2)$ that is however not closed in the operator norm. Passing to the closure leads to the Banach algebra of all so-called {\sl band-dominated operators} on $\ell^2$ -- a much larger operator class, where the matrix entries need not vanish, only decay in a certain way with their distance from the main diagonal. So every band-dominated operator $A$ is the norm limit of a sequence $(A_n)$ of band operators, whence $A=A_n+T_n$ with $\|T_n\|\to 0$. From here one can work with the formula 
\begin{equation} \label{eq:perturb}
\speps A\ =\ \speps(A_n+T_n)\ \subset\ \specn_{\eps+\|T_n\|}A_n
\end{equation}
with clever choices of $A_n\in BO(E)$ or even just $A_n\in\W(E)$ and $T_n$, see \cite{CW.Heng.ML:JST} for details.

%

\medskip

{\bf Acknowledgements.} The authors are grateful for inspiring conversations with Siegfried Beckus, Matt Colbrook, Mark Embree, Paul Hege, and Mattes Wittig.


\vfill
\noindent {\bf Author's addresses:}\\[2mm]
Simon~N.~Chandler-Wilde \hfill \href{mailto:s.n.chandler-wilde@reading.ac.uk}{\tt s.n.chandler-wilde@reading.ac.uk}\\
University of Reading\\
Department of Mathematics and Statistics\\
Reading, RG6 6AX\\
UK\\[5mm]
Marko Lindner\hfill \href{mailto:lindner@tuhh.de}{\tt lindner@tuhh.de}\\
Hamburg University of Technology (TUHH)\\
Institute of Mathematics\\
Am Schwarzenberg-Campus 3\\
D-21073 Hamburg\\
Germany\\[5mm]
Christian Seifert\hfill \href{mailto:christian.seifert@tuhh.de}{\tt christian.seifert@tuhh.de}\\
Hamburg University of Technology (TUHH)\\
Institute of Mathematics\\
Am Schwarzenberg-Campus 3\\
D-21073 Hamburg\\
Germany\\
and\\
University of the Free State\\
Mathematics and Applied Mathematics \\
Bloemfontein 9300\\
Republic of South Africa
\end{document}